\documentclass[11pt]{article}

\usepackage{latexsym}
\usepackage{amssymb}
\usepackage{amsthm}
\usepackage{amscd}
\usepackage{amsmath}
\usepackage{mathrsfs}
\usepackage[all]{xy}
\usepackage{hyperref} 
\usepackage[usenames,dvipsnames]{color}
\usepackage{graphicx,eepic}
\usepackage{float}

\usepackage{color}

\theoremstyle{definition}
\newtheorem* {theorem*}{Theorem}

\theoremstyle{definition}
\newtheorem{theorem}{Theorem}[section]

\theoremstyle{definition}

\theoremstyle{definition}
\newtheorem{observation}{Observation}[section]

\theoremstyle{definition}

\newtheorem{lemma}{Lemma}[section]
\theoremstyle{definition}

\theoremstyle{definition}

\theoremstyle{definition}

\newtheorem{conjecture}{Conjecture}[section]
\newtheorem{proposition}{Proposition}[section]
\newtheorem{corollary}{Corollary}[section]
\newtheorem* {remark}{Remark}
\newtheorem{example}{Example}[section]
\theoremstyle{definition}

\theoremstyle{definition}
\newtheorem* {remarks}{Remarks}

\xyoption{dvips}

\usepackage{fullpage}

\numberwithin{equation}{section}

\newcommand{\larc}[1]{\hspace{-.4ex}\overset{#1}{\frown}\hspace{-.4ex}}

\def\({\left(}
\def\){\right)}
   \newcommand{\FF}{\mathbb{F}}  \newcommand{\CC}{\mathbb{C}}  \newcommand{\QQ}{\mathbb{Q}}   \newcommand{\cP}{\mathcal{P}} 
   
  \newcommand{\cS}{\mathcal{S}} 
\newcommand{\cJ}{\mathcal{J}}

\def\NN{\mathbb{N}}
    \def\ZZ{\mathbb{Z}} \def\Aut{\mathrm{Aut}}    \def\GL{\mathrm{GL}}       \def\spanning{\textnormal{-span}}   
\def\Irr{\mathrm{Irr}}  \def\wt{\widetilde}

   \newcommand{\fkn}{\mathfrak{n}}
\newcommand{\cM}{\mathcal{M}}\newcommand{\h}{\mathfrak{h}}

\def\fk{\mathfrak}

\def\barr{\begin{array}}
\def\earr{\end{array}}
\def\ba{\begin{aligned}}
\def\ea{\end{aligned}}
\def\be{\begin{equation}}
\def\ee{\end{equation}}
\def\cS{\mathcal{S}}

\def\UT{\mathrm{UT}}
\def\fkt{\fk{u}}

\makeatletter
\renewcommand{\@makefnmark}{\mbox{\textsuperscript{}}}
\makeatother

\allowdisplaybreaks[1]

\UseCrayolaColors

\begin{document}
\title{Combinatorial methods of character enumeration for the unitriangular group}
\author{Eric Marberg\footnote{This research was conducted with government support under
the Department of Defense, Air Force Office of Scientific Research, National Defense Science
and Engineering Graduate (NDSEG) Fellowship, 32 CFR 168a.} \\ Department of Mathematics \\ Massachusetts Institute of Technology \\ \tt{emarberg@math.mit.edu}}
\date{}

\maketitle

\setcounter{tocdepth}{2}

\begin{abstract}
Let $\UT_n(q)$ denote the  group of unipotent $n\times n$ upper triangular matrices over a field with $q$ elements.  The degrees of the complex irreducible characters of $\UT_n(q)$ are precisely the integers $q^e$ with $0\leq e\leq \lfloor \frac{n}{2} \rfloor \lfloor \frac{n-1}{2} \rfloor$, and it has been conjectured that the number of irreducible characters of $\UT_n(q)$ with degree $q^e$ is a polynomial in $q-1$ with nonnegative integer coefficients (depending on $n$ and $e$). 
We  confirm this conjecture when $e\leq 8$ and $n$ is arbitrary by a computer calculation.  In particular, we describe an algorithm which allows us to derive explicit bivariate polynomials in $n$ and $q$ giving the number of irreducible characters of $\UT_n(q)$ with degree $q^e$ when $n>2e$ and $e\leq 8$.  When divided by $q^{n-e-2}$ and written in terms of the variables $n-2e-1$ and $q-1$, these functions are actually bivariate polynomials with nonnegative integer coefficients, suggesting an even stronger conjecture concerning such character counts.   As an application of these calculations, we are able to show that  all irreducible characters of $\UT_n(q)$ with degree $\leq q^8$ are Kirillov functions.  
We also discuss some related results concerning the problem of counting the irreducible constituents of individual supercharacters of $\UT_n(q)$.
  \end{abstract}

\section{Introduction}

Let $\FF_q$ be a finite field with $q$ elements and write $\UT_n(q)$ to denote the \emph{unitriangular group} of $n\times n$ upper triangular matrices over $\FF_q$ with all diagonal entries equal to 1.  This is a Sylow $p$-subgroup of the general linear group $\GL(n,\FF_q)$, where $p>0$ is the characteristic of $\FF_q$.  
This work concerns the problem of counting the irreducible characters of $\UT_n(q)$.

By a result of Isaacs \cite{I95}, the degrees of all (complex) irreducible characters of $\UT_n(q)$ are powers of $q$.  In fact, Huppert \cite{Huppert} has shown that the set of integers occurring as degrees of irreducible characters of $\UT_n(q)$ is $\{ q^e : 0\leq e \leq \cM_n\}$, where 
 \[ \cM_n = \lfloor \tfrac{n}{2} \rfloor \lfloor \tfrac{n-1}{2} \rfloor= \begin{cases} m(m-1),&\text{if $n=2m$,} \\  m^2,&\text{if }n=2m+1.\end{cases}\]   
One may therefore define  $N_n(q)$ and $N_{n,e}(q)$  for each positive integer $n$,  prime power $q>1$, and  integer $e$, as the numbers  
\[\ba
 N_n(q) &=\text{the number of irreducible characters (also, of conjugacy classes) of $\UT_n(q)$,} 
\\
 N_{n,e}(q) &=\text{the number of irreducible characters of $\UT_n(q)$ of degree $q^e$}.\ea\]   There is significant interest in the nature of these functions, in particular concerning whether or not they are polynomials in $q$. 
 Higman \cite{Higman} conjectured in 1960 that for each fixed $n$, the function $N_n(q)$ is a polynomial in $q$, and fourteen years later, Lehrer \cite{Lehrer} conjectured similarly that each function $N_{n,e}(q)$ is a polynomial in $q$.  Lehrer's conjecture certainly implies Higman's.
  More recently, Isaacs \cite{I} in 2007 put forth the even stronger conjecture that $N_{n,e}(q)$ is a polynomial function in $q-1$ with nonnegative integer coefficients. 

In the past fifteen years, a number of researchers have made significant progress in studying these conjectures.   Vera-Lopez and Arregi \cite{VeraLopez13} developed an algorithm to enumerate the conjugacy classes of $\UT_n(q)$ and used this to verify Higman's conjecture for $n\leq 13$ in 2003.  A few years later, Isaacs \cite{I} gave conjectural polynomials for $N_{n,e}(q)$ with $n\leq 9$.   Evseev \cite{E}  has recently calculated polynomials in $q$ giving $N_{n,e}(q)$ for $n\leq 13$; his methods confirm Isaacs's formulas.
The polynomials $N_{n,e}(q)$ with $n\leq 13$ have nonnegative integer coefficients when written as functions of $q-1$ \cite[Proposition 1.6]{E}, confirming Isaacs's conjecture for these values of $n$.

The results just mentioned derive essentially from the development of increasingly robust algorithms for enumerating the irreducible characters of $\UT_n(q)$ and related groups.  By contrast, investigations of 
 the functions $N_{n,e}(q)$ when $e$ is fixed and $n$ is arbitrary have depended to a much greater extent on ad hoc, manual calculations.  
 In the late 1990s, Marjoram \cite{Marj1,Marj2} computed bivariate polynomials in $n$ and $q$ giving the number of irreducible characters of $\UT_n(q)$ of the three lowest and highest degrees.  In his paper \cite{I}, Isaacs contributes some additional formulas.   
More recently, Loukaki  \cite{L} has computed $N_{n,e}(q)$ when $0\leq e\leq 3$, and Le \cite{Le} has rederived Marjoram's formulas for $N_{n,e}(q)$ when $ \cM_n -2 \leq e \leq \cM_n$ (this part of Marjoram's work was never published and required $n$ to be even; Le removes this condition).  

This paper began as an application of some recent observations concerning the 
constituents of supercharacters of algebra groups.  
These results, combined with the methods developed by Evseev in \cite{E}, lead us to an algorithm for computing $N_{n,e}(q)$ for small values of $e$ (and $n,q$ arbitrary).  Using this algorithm,
we are able to verify Isaacs's conjecture$-$that $N_{n,e}(q)$ is a polynomial in $q-1$ with nonnegative integer coefficients$-$for $e\leq 8$.  The formulas we obtain for $N_{n,e}(q)$ display a striking pattern not at all apparent in the antecedent calculations undertaken in \cite{I,L,Marj2}.  The following theorem summarizes our observations.

\begin{theorem}\label{intro} 
If $e \in \{1,\dots,8\}$, then for all integers $n>2e$ and prime powers $q>1$,
\[ N_{n,e}(q) = q^{n-e-2}  \sum_{i=1}^{2e}\frac{c_{e,i}!}{e!} \cdot f_{e,i}(n-2e-1)\cdot  (q-1)^i\] where $c_{e,i} = \frac{1}{2} + | \frac{1}{2} +e -i|$ 
 and each $f_{e,i}(x)$ is a polynomial with nonnegative integer coefficients, such that $\frac{c_{e,i}!}{e!} \cdot f_{e,i}(x)$ is a nonnegative integer for all nonnegative integer values of $x$. 
 \end{theorem}

\begin{remark}
The case $e=0$ is notably excluded here; one can show without difficulty that $N_{n,0}(q) = q^{n-1}$.  Note that the values of $c_{e,i}$ for $i=1,\dots,2e$ are just the integers $e,e-1,\dots,1,1,\dots,e-1,e$.  
\end{remark}

We tabulate the polynomials $f_{e,i}(x)$ for $e \in \{1,\dots,8\}$ in an appendix.  The limiting factor in our calculations was simply their duration, and so it may be possible to push our methods further with some optimization. 
  Fascinatingly, the polynomials $f_{e,i}(x)$ for $e\in\{1,\dots,8\}$ have degrees $e+1-c_{e,i}=1,2,\dots,e ,e,\dots,2,1$ and their leading coefficients are 
   \[N(e,1),\ N(e,2),\ \dots,\ N(e,e),\ N(e,e),\ \dots, N(e,2),\ N(e,1)\] where $N(m,k) =\frac{1}{k} \binom{m-1}{k-1} \binom{m}{k-1}$ denotes the Narayana numbers (sequence A00126 in \cite{OEIS}).
The theorem and these observations evince a startling degree of order in the functions $N_{n,e}(q)$, suggesting the following  conjecture.

\begin{conjecture}\label{intro-conj}
Theorem \ref{intro} holds if $e$ is any positive integer.
\end{conjecture}

Write $p$ for the characteristic of $\FF_q$. An important reason for caution in considering this conjecture is the existence of ``exotic'' irreducible characters of $\UT_n(q)$ when $n \gg p$.  By ``exotic,'' we mean characters taking values outside the cyclotomic field $\QQ(\zeta_p)$ where $\zeta_p = e^{2\pi i / p}$ is a primitive $p$th root of unity; our paper \cite{supp1} describes an explicit construction giving examples of  such characters for all primes $p$.  This phenomenon is an artifact of small characteristic, for when $n<2p$, the irreducible characters of $\UT_n(q)$ all have values in $\QQ(\zeta_p)$ \cite[Corollary 12]{Sangroniz}.
All irreducible characters counted by our methods have values in $\QQ(\zeta_p)$, and there is no evidence to suggest that the numbers of ``exotic'' irreducible characters should have nice polynomial properties.  
Thus, at the very least, it may be more plausible to consider Conjecture \ref{intro-conj} with the additional condition that  the characteristic of $\FF_q$ be sufficiently large.

\def\Ad{\mathrm{Ad}}

Not only do all the characters counted by our methods have values in $\QQ(\zeta_p)$; in fact, we can prove that they are all \emph{Kirillov functions}. 
By this we mean functions on $\UT_n(q)$ given by the following construction.  
Let $\fkt_n(q)$ denote the algebra of $n\times n$ upper triangular matrices over $\FF_q$ with all diagonal entries equal to 0.
There is a coadjoint action of $\UT_n(q)$ on the irreducible characters of the additive abelian group $\fkt_n(q)$, given by $g \colon \vartheta \mapsto \vartheta \circ \Ad(g)^{-1}$ where $\Ad(g)(X) = gXg^{-1}$ for $g \in \UT_n(q)$ and $X \in \fkt_n(q)$.
If $\Omega$ is a coadjoint orbit, then the corresponding \emph{Kirillov function} $\psi \colon \UT_n(q) \to \QQ(\zeta_p)$ is the complex-valued function 
\[ \psi(g) =|\Omega|^{-1/2} \sum_{\vartheta \in \Omega} \vartheta(g-1),\qquad\text{for }g\in \UT_n(q).
\] Kirillov \cite{K} conjectured that these functions comprise all the irreducible characters of $\UT_n(q)$, and we observed in \cite{supp0} that a recent calculation of Evseev \cite{E} shows that this conjecture holds if and only if $n\leq 12$. Here we prove an analogous but less precise result:

\begin{theorem}\label{intro2}
Every irreducible character of $\UT_n(q)$ of degree $\leq q^8$ is a Kirillov function.
\end{theorem}

The upper bound of $q^8$ is likely not optimal, as 
the smallest known degree of an irreducible character of $\UT_n(q)$ not given by a Kirillov function is $q^{16}$ (see \cite{supp1}). 

We derive these results by considering the more general problem of enumerating the irreducible constituents of the \emph{supercharacters} of $\UT_n(q)$.
Discovered by Andr\'e \cite{Andre,Andre1}, the supercharacters of $\UT_n(q)$ are a family of often reducible characters whose irreducible constituents partition the set of all irreducible characters of the group.  Analogous to the way that each irreducible character of the symmetric group $S_n$ has a shape given by a partition of $n$, each supercharacter of $\UT_n(q)$ has a shape given by a set partition of $[n] \overset{\mathrm{def}} = \{1,2,\dots,n\}$.  The number of supercharacters of $\UT_n(q)$ with a given shape is a power of $q-1$, and the group of automorphisms of $\UT_n(q)$ induced by the diagonal subgroup of $\GL(n,\FF_q)$ 
acts transitively on the set of such
supercharacters \cite[Observation 3.1]{supp1}.
For each positive integer $n$,  prime power $q$, integer $e$, and set partition $\Lambda$ of $[n]$,   we may thus define  $N_\Lambda(q)$ and $N_{\Lambda,e}(q)$ as the nonnegative numbers 
\[\ba N_\Lambda(q)=\ & \text{the number of irreducible constituents of any supercharacter of $\UT_n(q)$ with shape  $\Lambda$}, \\
N_{\Lambda,e}(q) =\ &\text{the number of irreducible constituents  of degree $q^e$ 
of any supercharacter of $\UT_n(q)$} \\ &\text{with shape $\Lambda$}.
\ea\] 
%
\def\fkcr{\mathfrak{C}}
Within this framework,  our main results are as follows:
\begin{enumerate}
\item[(a)] We describe a simple construction which attaches to each set partition $\Lambda$ a nilpotent $\FF_q$-algebra $\wt\fkcr_\Lambda(q)$ generated as a vector space by $\Lambda$ and its crossings.  

\item[(b)] We show that $N_{\Lambda,e}(q)$ counts the number of irreducible representations of the corresponding algebra group $1+\wt\fkcr_\Lambda(q)$ with a certain degree and central character.  Evseev's {\sc{Magma}} implementation \cite{implement} of the algorithm he describes in \cite{E} may be used to compute these counts as functions in $q$.

\item[(c)] We define a decomposition of a set partition into ``connected components'' and prove using (b) that $N_{\Lambda,e}(q)$ factorizes according to this decomposition.  (There exists already a notion of a connected set partition; we define something slightly more restrictive which we call \emph{crossing-connected}.)

\item[(d)] It follows from (c) that the functions $N_{\Lambda,e}(q)$ are completely determined by the cases where $\Lambda$ is crossing-connected. 
We show more strongly that for small values of $e$, only a finite number of crossing-connected set partitions $\Lambda$ have $N_{\Lambda,e}(q) \neq 0$.  This allows us to compute $N_{n,e}(q)$ with $e$ fixed and $n,q$ arbitrary using (b) and (c).

\end{enumerate}
These items will allow us to prove Theorems \ref{intro} and \ref{intro2}.  Our calculations suggest as well the following analogue of Lehrer's conjecture:
\begin{conjecture}\label{ours} 
%
For each set partition $\Lambda \vdash[n]$ and integer $e\geq 0$, the function $N_{\Lambda,e}(q)$ is a polynomial in $q$ with integer coefficients.
\end{conjecture}
As before, we lack much evidence that this statement should be true in general, given the existence of irreducible characters of $\UT_n(q)$ with values in arbitrarily large cyclotomic fields.  
However, using (a)-(d),  we will verify this conjecture for $n\leq 6$ by inspection and for $n\leq 13$ via a computer calculation.  In so doing, we will discover that the analog of Isaacs's conjecture for $N_{\Lambda,e}(q)$ does not hold: there are integers $e$ and set partitions $\Lambda$ of $[n]$ when $n\geq 13$ for which $N_{\Lambda,e}(q)$ is a polynomial in $q-1$ with both positive and negative integer coefficients.

\subsection*{Acknowledgements}

I  thank the anonymous referees for their helpful remarks and suggestions.

\section{Preliminaries}
\label{prelim-sect}

Here we briefly  establish our notational conventions and discuss in slightly greater detail    the constructions mentioned in the introduction.


Given a finite group $G$, we let $\langle\cdot,\cdot\rangle_G$ denote the standard inner product on the complex vector space of functions $G \to \CC$, defined by 
$ \langle f,g\rangle_G = \frac{1}{|G|} \sum_{x \in G} f(x) \overline{g(x)}$.  If the context is clear, we usually omit the subscript. 
Write $\Irr(G)$ for the set of complex irreducible characters of $G$, or equivalently the set of characters $\chi$ of $G$ with $\langle \chi,\chi \rangle_G = 1$.  A function $ G\to \CC$ is then a character if and only if it is a nonzero sum of irreducible characters with nonnegative integer coefficients.
A character $\psi$ is a {constituent} of another character $\chi$ if $\chi-\psi$ is a character or zero; in this case, the largest integer $m$ such that $\chi-m\psi$ is a character or zero is the {multiplicity} of $\psi$ in $\chi$.
 %
%

 Given integers $1\leq i<j \leq n$, we let 
\[ \ba e_{ij} & = \text{the matrix in $\fkt_n(q)$ with 1 in position $(i,j)$ and zeros elsewhere,}\\
e_{ij}^* &=\text{the $\FF_q$-linear map $\fkt_n(q)\to \FF_q$ given by $e_{ij}^*(X) = X_{ij}$.}\ea\]
These matrices and maps are then dual bases of $\fkt_n(q)$ and its dual space $\fkt_n(q)^*$.


\subsection{Algebra groups}
\label{alg}

While we are mostly concerned with $\UT_n(q)$, it is helpful to present a few preliminary definitions in the greater generality of algebra groups.

Let $\fkn$ be a (finite-dimensional, associative) nilpotent $\FF_q$-algebra, and $\fkn^*$ its dual space of $\FF_q$-linear maps $\fkn \to \FF_q$.  Write $G = 1+\fkn$ to denote the corresponding \emph{algebra group}; this is the set of formal sums $1+X $ with $X \in \fkn$, made into a group via the multiplication \[(1+X)(1+Y) = 1+X+Y+XY.\]  The algebra $\fkt_n(q)$ of strictly upper triangular $n\times n$ matrices over $\FF_q$ and the {unitriangular group} $\UT_n(q) = 1 + \fkt_n(q)$ serve as prototypical examples of $\fkn$ and $G$.

We call a subgroup of $G=1+\fkn$ of the form $H = 1+\h$ where $\h \subset \fkn$ is a subalgebra an \emph{algebra subgroup}.  By theorems of Isaacs \cite{I} and Halasi \cite{H}, every irreducible representation of an algebra group over $\FF_q$ has $q$-power degree and is obtained by inducing a linear representation of an algebra subgroup.
If $\h \subset \fkn$ is a two-sided ideal then $H$ is a normal algebra subgroup of $G$, and the map $gH \mapsto 1+(X+\h)$ for  $g=1+X \in G$ gives an isomorphism $G/H \cong 1 + \fkn /\h$.  In practice we usually identify the quotient $G/H$ with the algebra group $1 + \fkn /\h$ by way of this canonical map.

\subsection{Kirillov functions and supercharacters}\label{supercharacters}

Fix a nontrivial homomorphism $\theta \colon \FF_q^+\to \CC^\times$ from the additive group of $\FF_q$ to the multiplicative group of nonzero complex numbers.  Observe that $\theta$ takes values in the cyclotomic field $\QQ(\zeta_p)$, where $p>0$ is the characteristic of $\FF_q$.
For each $\lambda \in \fkn^*$, we define $\theta_\lambda \colon G \to \QQ(\zeta_p)$ as the function 
\be\label{theta_lambda} \theta_\lambda(g) = \theta\circ \lambda(g-1),\qquad\text{for }g \in G.\ee
The maps $\theta\circ \lambda \colon \fkn \to \CC$ are the distinct irreducible characters of the abelian group $\fkn$, and from this it follows that the functions $\theta_\lambda \colon G\to \CC$ are an orthonormal basis (with respect to $\langle\cdot,\cdot\rangle_G$) for all functions $G\to \CC$.   

The most generic methods we have at our disposal for constructing characters of algebra groups involve summing the functions $\theta_\lambda$ over orbits in $\fkn^*$ under an appropriate action of $G$.
Kirillov functions provide perhaps the most natural example of such a construction.  Their definition relies on the \emph{coadjoint} action of $G$ on $\fkn^*$, by which we mean the left action $(g,\lambda) \mapsto g\lambda g^{-1}$ where we define
\[ 
(g \lambda g^{-1})(X) = \lambda(g^{-1} X g),\qquad\text{for } \lambda \in \fkn^*,\ g\in G,\ X \in \fkn.\]  Denote  the coadjoint orbit of $\lambda \in \fkn^*$ by $\lambda ^G$.    
The \emph{Kirillov function} $\psi_\lambda$ indexed by $\lambda \in \fkn^*$ is then the map  $G \to \QQ(\zeta_p)$ defined by 
\be\label{kirillov-def} \psi_\lambda = \frac{1}{\sqrt{|\lambda^G|}} \sum_{\mu \in \lambda^G} \theta_\mu.\ee The size of $\lambda^G$ is a power of $q$ to an even integer \cite[Lemma 4.4]{DI} and so $\psi_\lambda(1) = \sqrt{|\lambda^G|}$ is a nonnegative integer power of $q$. 
We have $\psi_\lambda = \psi_\mu$ if and only if $\mu \in \lambda^G$, and the distinct Kirillov functions on $G$ form an orthonormal basis (with respect to $\langle\cdot,\cdot\rangle_G$) for the class functions on the group.  Kirillov functions are sometimes but not always (irreducible) characters; for example,
$\Irr(\UT_n(q)) = \{ \psi_\lambda : \lambda \in \fkt_n(q)^*\}$ if and only if $n \leq 12$ \cite[Theorem 4.1]{supp0}.


While Kirillov functions provide an accessible orthonormal basis for the class functions of an algebra group,  supercharacters  alternatively provide an accessible family of orthogonal characters.  
Andr\'e \cite{Andre,Andre1} first defined these characters in the special case $G=\UT_n(q)$ as a practical substitute for the group's unknown irreducible characters.  Several years later, Yan \cite{Yan} showed how one could replace Andr\'e's definition with a more elementary construction, which Diaconis and Isaacs \cite{DI} subsequently generalized to algebra groups.

We define the supercharacters of $G=1+\fkn$ in a way analogous to Kirillov functions, but using left and right actions of $G$ on $\fkn^*$ in place of the coadjoint action.
In detail, the group $G$ acts on the left and right on $\fkn$  by multiplication, and on 
 $\fkn^*$ by $(g , \lambda) \mapsto g\lambda$ and $( \lambda,g) \mapsto \lambda g$ where we define
 \[ g\lambda(X) = \lambda(g^{-1}X)\qquad\text{and}\qquad \lambda g(X) = \lambda(Xg^{-1}),\qquad\text{for }\lambda \in \fkn^*,\ g \in G,\ X \in \fkn.\]
 These actions commute, in the sense that $(g\lambda) h = g(\lambda h)$ for $g,h \in G$, so there is no ambiguity in removing all parentheses and writing expressions like $g\lambda h$.   We denote the left, right, and two-sided orbits of $\lambda \in \fkn^*$ by  $G\lambda$, $\lambda G$, and $G\lambda G$ .
Notably, $G\lambda$ and $\lambda G$ have the same cardinality and $|G\lambda G| = \frac{|G\lambda||\lambda G|}{|G\lambda \cap \lambda G|}$ \cite[Lemmas 3.1 and 4.2]{DI}. %
The \emph{supercharacter} $\chi_\lambda$ indexed by $\lambda \in \fkn^*$ is the function  $G \to \QQ(\zeta_p)$ defined by 
\be\label{superchar-def} 
 \chi_\lambda = \frac{|G\lambda|}{|G\lambda G|} \sum_{\mu \in G \lambda G} \theta_\mu.
\ee 
Supercharacters are always characters but often reducible.  We have $\chi_\lambda = \chi_\mu$ if and only if $\mu \in G\lambda G$, and every irreducible character of $G$ appears as a constituent of a unique supercharacter.  The orthogonality of the functions $\theta_\mu$ implies that \[\langle \chi_\lambda,\chi_\mu\rangle_G = \begin{cases} |G\lambda \cap \lambda G|, &\text{if }\mu \in G \lambda G , \\
0,&\text{otherwise,}\end{cases}
\qquad\text{for }\lambda,\mu \in \fkn^*.\] 
If $\chi_\lambda$ is irreducible, then $\chi_\lambda = \psi_\lambda$ is a Kirillov function.  Furthermore, $\frac{|G\lambda|}{|G\lambda \cap \lambda G|} \chi_\lambda$ is the character of a two-sided ideal in $\CC G$, so all irreducible constituents of $\frac{|G\lambda|}{|G\lambda \cap \lambda G|} \chi_\lambda$ have multiplicity equal to their degree.  For proofs of these facts, see \cite{DI}.

\subsection{Constituents of supercharacters}\label{constits}

\def\Kir{\mathrm{Kir}}

  The supercharacter $\chi_\lambda$ is a positive integral linear combination of the Kirillov functions indexed by functionals in the two-sided orbit $G\lambda G$ \cite[Theorem 5.7]{AndreAdjoint}.  Let $\Kir(G,\chi_\lambda)$ denote the set of such constituent Kirillov functions:
\[ \Kir(G,\chi_\lambda) = \{ \psi_\mu : \mu \in G\lambda G \} = \{ \psi_\mu : \mu \in \fkn^*\text{ such that }\langle \chi_\lambda,\psi_\mu\rangle \neq 0\}.\]
Likewise, let $\Irr(G,\chi_\lambda)$ denote the set of irreducible characters which are constituents of $\chi_\lambda$.  By \cite[Theorem 2.1]{O}, $\Irr(G,\chi_\lambda)$ and $\Kir(G,\chi_\lambda)$ have the same cardinality, and one naturally asks when these two sets are equal.  Then following lemma is useful in answering this question.

\begin{lemma}\label{obs-las}
If $\alpha,\beta,\chi$ are supercharacters of an algebra group $G$ such that $\chi = \alpha \otimes \beta$ and $\langle\chi,\chi\rangle = \langle \alpha,\alpha\rangle \langle \beta,\beta\rangle$, then the maps
\[ \barr{ccc} \Irr(G,\alpha) \times \Irr(G,\beta) & \to & \Irr(G,\chi) \\ 
(\psi,\psi') &\mapsto & \psi \otimes \psi'\earr\qquad\text{and}\qquad\barr{ccc} \Kir(G,\alpha) \times \Kir(G,\beta) & \to & \Kir(G,\chi) \\ 
(\psi,\psi') &\mapsto & \psi \otimes \psi'\earr\] 
are both bijections.  Consequently, if in this setup $\Irr(G,\alpha) = \Kir(G,\alpha)$ and $\Irr(G,\beta) = \Kir(G,\beta)$, then $\Irr(G,\chi) = \Kir(G,\chi)$.
\end{lemma}

\begin{proof}
Suppose $\alpha$ and $\beta$ decompose into positive integral linear combinations of distinct irreducible characters as
$ \alpha = a_1\phi_1 +\dots + a_r \phi_r$ and $\beta = b_1 \psi_1 + \dots + b_s \psi_s$ for positive integers $a_i,b_i$ and  $\phi_i,\psi_i \in \Irr(G)$.   Since inner products of characters are nonnegative integers, one computes 
\[ \barr{c} \langle \chi,\chi \rangle \geq \sum_{i,j} a_i^2 b_j^2 \langle \phi_i \otimes \psi_j, \phi_i \otimes \psi_j \rangle \geq \sum_{i,j} a_i^2 b_j^2 =
 \langle \alpha,\alpha\rangle \langle \beta,\beta \rangle. \earr\] 
  By hypothesis we have equality throughout, which implies that $\langle \phi_i \otimes \psi_j, \phi_{i'} \otimes \psi_{j'} \rangle =\delta_{ii'} \delta_{jj'}$ which is in turn equivalent to the first map being a bijection.
  
 Products of Kirillov functions decompose as nonnegative integral linear combinations of Kirillov functions: \cite[Theorem 5.5]{AndreAdjoint} asserts that this is true of restrictions of Kirillov functions to algebra subgroups, and the result for products follows by considering the restriction from $G\times G$ to its diagonal subgroup, as in the proof of \cite[Theorem 6.6]{DI}. Since, as noted above, a supercharacter is also a positive integral linear combination of its constituent Kirillov functions, the same argument shows that our second map is bijection.
 \end{proof}

Helpfully, 
the problem of enumerating the irreducible constituents of a supercharacter  reduces to that of counting the irreducible representations with a certain central character of a quotient of a typically much smaller algebra subgroup.  We shall find in Section \ref{crossing} that for $G = \UT_n(q)$, the structure of this quotient is closely related to combinatorial features of the set partition of $[n]$ giving the shape of the supercharacter under examination.
To describe this reduction precisely, for each $\lambda \in \fkn^*$,  define three subspaces $\fk k_\lambda,\fk l_\lambda, \fk s_\lambda \subset\fkn$ by 
\[ \ba 
\fk k_\lambda &= \{ X \in \fkn :\lambda(X)= \lambda(XY) =0\text{ for all }Y \in \fkn\}, \\
\fk l_\lambda &= \{ X \in \fkn : \lambda(XY) =0\text{ for all }Y \in \fkn\},\\
\fk s_\lambda &= \{ X \in \fkn : \lambda(XY) =0\text{ for all }Y\in \fk l_\lambda\}.
\ea\]  Alternatively, one constructs $\fk l_\lambda$ as the left kernel of the bilinear form 
$ B_\lambda \colon \fkn \times \fkn \to  \FF_q$ given by $(X,Y)\mapsto\lambda(XY)$; $\fk s_\lambda$ as the left kernel of the restriction of $B_\lambda$ to the domain $\fkn \times \fk l_\lambda$; and $\fk k_\lambda$ as the intersection $\fk l_\lambda \cap \ker \lambda$.  
The subspace $\fk s_\lambda$ is a subalgebra of $\fkn$; the subspace $\fk l_\lambda$ is a right ideal of $\fkn$ and a two-sided ideal of $\fk s_\lambda$; and the subspace $\fk k_\lambda$ is a two-sided ideal in $\fk s_\lambda$ (see Section 3.1 in \cite{supp0}).
We therefore may define $K_\lambda$, $L_\lambda$, $S_\lambda \subset G$ as the corresponding algebra subgroups $K_\lambda = 1+\fk k_\lambda$, 
$L_\lambda = 1+\fk l _\lambda$, $S_\lambda = 1+\fk s_\lambda$. Both $K_\lambda$ and $L_\lambda$ are the normal in  $S_\lambda$.  Our groups $L_\lambda$ and $S_\lambda$ are the same as the ones defined in Section 4.1 in \cite{AndreAdjoint}; see \cite[Lemma 3.1]{supp0} for a proof of this fact.

These algebra groups relate to the irreducible constituents of $\chi_\lambda$ in the following way.  Since $\fk k_\lambda$ and $\fk l_\lambda$ are two-sided ideals in $\fk s_\lambda$, we may identify $S_\lambda/ K_\lambda$ and $L_\lambda/K_\lambda$ with the algebra groups $1 + \fk s_\lambda / \fk k_\lambda$ and $1 + \fk l_\lambda / \fk k_\lambda$.  Let \[\pi \colon 1+X \mapsto 1 + (X+\fk k_\lambda)\] denote the quotient homomorphism $S_\lambda \to S_\lambda / K_\lambda$.
The following result combines Theorem 3.2 and Corollary 4.1 in \cite{supp0}.

\begin{theorem}\label{count} Let $\fkn$ be a finite-dimensional associative nilpotent $\FF_q$-algebra, write $G=1+\fkn$, and let $\lambda \in \fkn^*$. 
Then the number of irreducible constituents of degree $q^e$ of the supercharacter $\chi_\lambda$  is equal to the number of irreducible characters $\psi$ of $S_\lambda/K_\lambda$ such that 
\be\label{condition}
\psi\circ \pi(z) =\frac{\langle \chi_\lambda,\chi_\lambda\rangle}{\chi_\lambda(1)} \cdot q^e\cdot   \theta_\lambda(z),\qquad\text{for all }z \in L_\lambda.
\ee
Furthermore, every irreducible
constituent  with degree $q^e$ of the supercharacter $\chi_\lambda$ is a Kirillov function if and only if every irreducible character satisfying (\ref{condition}) of the algebra group $S_\lambda /K_\lambda \cong1 + \fk s_\lambda / \fk k_\lambda$  is a Kirillov function.
\end{theorem}

\subsection{Supercharacters of $\UT_n(q)$}

Fix a positive integer $n$ and a prime power $q>1$, and let $\fkt_n(q)^*$ denote the set of $\FF_q$-linear maps $\fkt_n(q) \to \FF_q$.  
Recall that a matrix is said to be \emph{monomial} if it has exactly one nonzero entry in each row and column.  Following \cite{Sangroniz}, we say that a matrix is \emph{quasi-monomial} if it has at most one nonzero entry in each row and column.  Given $\lambda \in \fkt_n(q)^*$ and integers $i,j$, let
\[ \lambda_{ij} =\begin{cases} \lambda(e_{ij}),&\text{if }1\leq i<j\leq n, \\  0,&\text{otherwise}.\end{cases}\]
We define $\lambda \in \fkt_n(q)^*$ to be quasi-monomial if the matrix $\sum_{i,j \in [n]} \lambda_{ij} e_{ij} \in \fkt_n(q)$ is quasi-monomial. 

\def\SC{\mathrm{SC}}

The following important fact is due originally to Andr\'e \cite{Andre,Andre1} and Yan \cite{Yan}: the quasi-monomial maps $\lambda \in \fkt_n(q)^*$  index the distinct supercharacters of $\UT_n(q)$; i.e., these elements represent the distinct two-sided $\UT_n(q)$-orbits in $\fkt_n(q)^*$ and  
\[\barr{ccc}   \Bigl\{{ \text{Quasi-mononomial maps } \lambda \in \fkt_n(q)^*}\Bigr \} &\to&\Bigl\{{\text{Supercharacters of $\UT_n(q)$}}\Bigr\} \\
\lambda&\mapsto& \chi_\lambda\earr\] is a bijection.
%
%
%
There is an especially simple product formula for the supercharacters of this group; see, for example, Section 2.3 in \cite{Thiem}.  For our purposes, only the following consequence of this formula will be needed:

\begin{lemma}\label{form} 



If $\lambda \in \fkt_n(q)^*$ is quasi-monomial and $\alpha,\beta \in \fkt_n(q)^*$ are such that $\lambda = \alpha+\beta$ and $\alpha_{ij} \beta_{ij} = 0$ for all $i,j$, then $\chi_\lambda = \chi_\alpha \otimes \chi_\beta$.

\end{lemma}

\begin{proof}
Given integers $1\leq i<j \leq n$ and $t \in \FF_q^\times$, let $\chi^{i\larc{t}j}$ be the supercharacter of $\UT_n(q)$ indexed by $te_{ij}^* \in \fkt_n(q)^*$.
Let $A(\mu)$ for $\mu \in \fkt_n(q)^*$ be the set of triples $(i,j,t) \in \ZZ\times \ZZ \times \FF_q^\times$ with
 $\mu_{ij} =t\neq 0$.   Thiem notes  just after \cite[Eq.\ (2.2)]{Thiem} that the product formula for $\chi_\lambda$ shows that $\chi_\lambda = \prod_{(i,j,t) \in A(\lambda)} \chi^{i\larc{t} j}$.
This identity is enough to conclude $\chi_\lambda = \chi_\alpha \otimes \chi_\beta$, since our
 hypotheses imply that $\alpha$ and $\beta$ are both quasi-monomial and that $A(\lambda)=A(\alpha)\cup A(\beta)$ is a disjoint union. 
 \end{proof}


\def\bar{\hspace{0.5mm}|\hspace{0.5mm}}
\def\st{\mathrm{st}}
\def\comp{\mathrm{Comp}}
\def\crcomp{\mathrm{CrComp}}
\def\Arc{\mathrm{Arc}}
\def\Cr{\mathrm{Cr}}
\def\cS{\mathcal{S}}


Each quasi-monomial $\lambda \in \fkt_n(q)^*$ naturally corresponds to a set partition of $[n]$, which we call its shape.  Recall that a {set partition} is just a set of pairwise disjoint, nonempty sets, and that the elements of a set partition are its \emph{parts}.  We write $\Lambda \vdash \cS$ to indicate that $\Lambda$ is a set partition, the union of whose parts is $\cS$.  
The number of set partitions of a set with $n$ elements is the Bell number $B_n$, which one can compute by the recurrence $B_{n+1} = \sum_{k=0}^{n} \binom{n}{k} B_k$ with $B_0=1$.

Formally, we define the \emph{shape} of a quasi-monomial $\lambda \in \fkt_n(q)^*$ as the finest set partition of $[n]$ in which $i,j$ belong to the same part whenever $\lambda_{ij} \neq 0$.  
Alternatively, the shape of $\lambda$ is the set partition whose parts are the vertex sets of the weakly connected components of the (weighted, directed) graph whose adjacency matrix is $\(\lambda_{ij}\)$.
For example, if $a_i\in \FF_q^\times$ then
\[ 
\ba
&a_1e_{1,2}^*  + a_2 e_{2,3}^*+\dots + a_5 e_{5,6}^*   \in \fkt_6(q)^*&& \quad\text{has shape}\qquad \{\{1,2,3,4,5,6\}\}\vdash[6], \\
& a_1e_{1,3}^* +a_2 e_{2,4}^*  +a_3e_{3,5}^*  \in \fkt_6(q)^*&& \quad\text{has shape}\qquad \{\{1,3,5\},\{2,4\}, \{6\}\}\vdash[6], \\
&0  \in \fkt_6(q)^*&& \quad\text{has shape}\qquad \{\{1\},\{2\},\{3\},\{4\},\{5\},\{6\}\}\vdash[6]
.
\ea\]  
%
%
%
The shape of a supercharacter  of $\UT_n(q)$ is by definition the shape of its unique quasi-monomial index $\lambda \in \fkt_n(q)^*$.
The map which associates to each supercharacter of $\UT_n(q)$ its shape defines a surjection
\[\label{thatmap} \barr{ccc} \Bigl\{ \text{Supercharacters of $\UT_n(q)$}\Bigr \} &\to
&
\Bigl\{{ \text{Set partitions of $[n]$} }\Bigr\}.\earr\]
This is a bijection if and only if $q=2$, and the inverse image of any  $\Lambda\vdash[n]$ has cardinality $(q-1)^{n-\ell(\Lambda)}$ where $\ell(\Lambda)$ is the number of parts of $\Lambda$.  
%
%


Define the functions  $N_\Lambda(q)$ and $N_{\Lambda,e}(q)$ as  in the introduction.
The total number $N_{n}(q)$ of irreducible characters of $\UT_n(q)$ and the total  number $N_{n,e}(q)$ of irreducible characters of degree $q^e$ are then given by
\be\label{imply} N_n(q) = \sum_{\Lambda \vdash[n]} (q-1)^{n-\ell(\Lambda)} N_\Lambda(q)\qquad\text{and}\qquad N_{n,e}(q) = \sum_{\Lambda\vdash[n]} (q-1)^{n-\ell(\Lambda)} N_{\Lambda,e}(q),\ee
where $\ell(\Lambda)$ is the number of parts of $\Lambda \vdash[n]$.  Thus, Higman's conjecture (that $N_n(q)$ is a polynomial in $q$) would follow if each $N_\Lambda(q)$ were a polynomial function in $q$, and similarly Lehrer's conjecture (that $N_{n,e}(q)$ is a polynomial in $q$) would hold if each $N_{\Lambda,e}(q)$ were a polynomial in $q$.  

 If $\lambda \in \fkt_n(q)^*$ is quasi-monomial with shape $\Lambda$ then $N_\Lambda(q)$ is the number of coadjoint orbits in the two-sided $\UT_n(q)$-orbit of $\lambda$ by \cite[Theorem 2.5]{O}.   In fact, it follows by \cite[Corollary 12]{Sangroniz} that if the characteristic of $\FF_q$ is sufficiently large, then $N_{\Lambda,e}(q)$ is the number of Kirillov functions $\psi$ with $\psi(1)= q^e$  and $\langle \psi,\chi_\lambda \rangle\neq 0$.  
Observations like this make it easy to believe that Conjecture \ref{ours} might fail if $n$ is sufficiently large and the characteristic of $\FF_q$ is sufficiently small.  We know from the results in \cite{supp1}, for example, that 
 there exist irreducible characters of $\UT_n(q)$ which are not Kirillov functions for large enough $n$, and nothing indicates that one should expect the numbers of Kirillov functions and irreducible characters of a certain degree to be equal.  Indeed, this does not hold for an algebra group in general: Jaikin-Zapirain constructs  in  \cite{JZ} an algebra group whose linear characters exceed in number its linear Kirillov functions.
Nevertheless, at present we have no data contradicting Conjecture \ref{ours} for $n\leq 13$.

\subsection{Notations for set partitions}

To describe methods of efficiently computing $N_{\Lambda,e}(q)$, it is useful to include a few more definitions pertaining to set partitions; for the most part we adopt our conventions from \cite{Stan} and \cite{Thiem}.  Throughout, $\cS$ denotes a finite subset of the natural numbers.

The \emph{standard representation} of a set partition $\Lambda\vdash \cS$ is the graph with vertex set $\cS$ which has an edge connecting $i,j \in \cS$ if  $j$ is the least integer greater than $i$ in the part of $\Lambda$ containing $i$. 
We denote by $\Arc(\Lambda)$ the set of pairs $(i,j) \in \cS^2$ with $i<j$ which are connected by an edge in the standard representation; we call this the \emph{arc set} of $\Lambda$. 
For example,
\be\label{std-ex} \Lambda= \{ \{ 1,3,4\} ,\{2,5\} \} \vdash[5]\quad\text{has standard representation}\quad \barr{c} \\ \xy<0.25cm,0.8cm> \xymatrix@R=-0.0cm@C=.5cm{
*{\bullet} \ar @/^1pc/ @{-} [rr]   & 
*{\bullet} \ar @/^1.2pc/ @{-} [rrr] &
*{\bullet} \ar @/^0.5pc/ @{-} [r]  &
*{\bullet} &
*{\bullet} \\
1   & 
2 &
3  &
4 &
5 
}\endxy
\earr\ee and $\Arc(\Lambda) = \{ (1,3), (2,5), (3,4)\}$.  
Observe that $\Arc(\Lambda)$ uniquely determines $\Lambda$  if the set $\cS$ which $\Lambda$ partitions is given.  Also, if $\lambda \in \fkt_n(q)^*$ is quasi-monomial with shape $\Lambda$, then $(i,j) \in \Arc(\Lambda)$ if and only if $\lambda_{ij} \neq 0$.  

We define a \emph{crossing} of $\Lambda\vdash\cS$ to be a 4-tuple $(i,j,k,l)\in \cS^4$ such that $i<j<k<l$ and $(i,k),(j,l)\in \Arc(\Lambda)$.
%
Intuitively, if one draws the standard representation of a set partition with all vertices collinear and all edges on the same side of the determined line, then each crossing corresponds to the intersection of two edges.  We denote by $\Cr(\Lambda)$ and $d(\Lambda)$ the following set and nonnegative integer:
\be\label{cr-d} \ba
&\barr{l} \Cr(\Lambda) = \{ (i,j) : (i,j,k,l) \text{ is a crossing of $\Lambda$ for some $k,l$}\},\earr\\
&\barr{l} d(\Lambda) = \sum_{(i,k) \in \Arc(\Lambda)} (k-i-1).\earr\ea\ee 
In the example (\ref{std-ex}), we have $\Cr(\Lambda) =\{(1,2)\}$ and $d(\Lambda) = 3$.   
%
%
The sets $\Arc(\Lambda)$ for $\Lambda \vdash[n]$ are the ``basic sets''  defined in \cite{AndreHecke}; in Andr\'e's notation, the set $\Cr(\Lambda)$ is precisely the set $\cS'$ \cite[Page 990]{AndreHecke} attached to the basic set $\Arc(\Lambda)$.

Of particular importance is the following standard fact \cite[Eqs.\ (2.2)-(2.3)]{Thiem}:
if $\chi$ is a supercharacter of $\UT_n(q)$ 
with shape $\Lambda$, then 
\be\label{fact} \langle \chi,\chi\rangle_{\UT_n(q)} = q^{|\Cr(\Lambda)|}\qquad\text{and}\qquad \chi(1)= q^{d(\Lambda)}.\ee
Thus the number of irreducible supercharacters of $\UT_n(q)$ is the number of {noncrossing set partitions} of $[n]$, which is well-known to be the Catalan number $C_n = \frac{1}{n+1} \binom{2n}{n}$.

Another noteworthy fact is the following result due to Andr\'e \cite{AndreHecke}.  Call a sequence $i_0 < i_1 < i_2 < \dots < i_k < i_{k+1} < i_{k+2}$ with every $(i_{r}, i_{r+2}) \in \Arc(\Lambda)$ a \emph{$k$-crossing} of $\Lambda\vdash\cS$; a crossing is then a 1-crossing.  A \emph{maximal crossing}  of $\Lambda$ with length $k$ is a $k$-crossing  which cannot be extended to a $(k+1)$-crossing.  The following is presented  as both \cite[Theorem 4]{AndreHecke} and  \cite[Theorem 7.3]{AndreAdjoint}.

\begin{theorem}\label{max-cross} Fix a positive integer $n$, a prime power $q>1$, and a set partition $\Lambda \vdash[n]$.  Then
$N_\Lambda(q)=1$ if and only if all maximal crossings of $\Lambda$ have even length, and in this case $N_{\Lambda,e}(q) =1$ for  $e = d(\Lambda) - \frac{1}{2} |\Cr(\Lambda)|$ and the unique irreducible constituent of any supercharacter with shape $\Lambda$ is a Kirillov function.
\end{theorem}

\section{Results}

In the following sections we establish items (a)-(d) in the introduction.
\subsection{Crossing algebras and character counts}\label{crossing}

Theorem \ref{count} shows that the numbers $N_{\Lambda,e}(q)$ count the irreducible representations with a certain central character of some quotient of algebra groups.  In this section we describe how this quotient 
corresponds to a natural algebra group structure on the crossing set $\Cr(\Lambda)$ of the set partition $\Lambda$.
  
This elementary construction is the following.
For each set partition $\Lambda$ and prime power $q>1$,  define the \emph{crossing algebra} $\fkcr_\Lambda(q)$ as the vector space $\fkcr_\Lambda(q) =\FF_q\spanning\{ e_{ij} : (i,j) \in \Cr(\Lambda)\} $ generated by the crossings of $\Lambda$, made into a nilpotent algebra via the multiplication 
\[ e_{ij} * e_{kl} = \begin{cases}  e_{il},&\text{if $j=k$ and $(i,l) \in \Cr(\Lambda)$,}\\ 
0,&\text{otherwise}, \end{cases}\qquad\text{for }(i,j),(k,l) \in \Cr(\Lambda).\]
Note that this product does not in general coincide with the usual matrix product $e_{ij} e_{kl} = \delta_{jk} e_{il}$. 
Likewise, we define $\wt \fkcr_\Lambda(q)$ to be the nilpotent $\FF_q$-algebra given as the central extension $\wt\fkcr_\Lambda(q) = \fkcr_\Lambda(q) \oplus \FF_q\spanning\{ z_\Lambda\}$ with multiplication
\be e_{ij} * z_\Lambda = z_\Lambda*e_{ij}=0\qquad\text{and}\qquad e_{ij} *e_{kl} =\begin{cases} e_{il},&\text{if $j=k$ and $(i,l) \in \Cr(\Lambda)$,}\\ 
z_\Lambda,&\text{if $j=k$ and $(i,l) \in \Arc(\Lambda)$,}\\
0,&\text{otherwise.}\end{cases}\ee

When $\Arc(\Lambda)  =\varnothing$ we have $\fkcr_\Lambda(q) = 0$ and $\wt\fkcr_\Lambda(q)\cong \FF_q$.  
To see that these algebras are always well-defined and associative, assume $\Arc(\Lambda) \neq \varnothing$  and let $\lambda$ be a nonzero multiple of $\sum_{(i,j) \in \Arc(\Lambda)} e_{ij}^* \in \fkt_n(q)^*$, so that $\lambda$ is quasi-monomial with shape $\Lambda$.  In the notation of Section \ref{constits}, we claim that 
\[ \fk C_\Lambda(q) \cong \fk s_\lambda / \fk l_\lambda\qquad\text{and}\qquad
\wt{\fk C}_\Lambda(q) \cong \fk s_\lambda / \fk k_\lambda.\]
  By definition $\fk k_\lambda$ is a codimension one subspace and hence an ideal of $\fk l_\lambda$; we noted in Section \ref{constits} that both $\fk k_\lambda$ and $\fk l_\lambda$ are  ideals in $\fk s_\lambda$; and  
Lemma 3.1 in \cite{supp1} asserts that  $\fk s_\lambda = \fk l_\lambda \oplus \FF_q\spanning\{ e_{ij} : (i,j) \in \Cr(\Lambda)\}$ as a vector space.  
These observations have the following consequences, which prove our claim:
\begin{enumerate}
\item[(1)] First, the cosets $e_{ij} + \fk l_\lambda$ for $(i,j) \in \Cr(\Lambda)$  form a basis for the quotient $\fk s_\lambda / \fk l_\lambda$, and it follows that the linear map defined by $e_{ij} + \fk l_\lambda \mapsto e_{ij} \in \fkcr_\Lambda(q)$ gives an algebra isomorphism $\fk s_\lambda /\fk l_\lambda \cong \fkcr_\Lambda(q)$.  

\item[(2)] Choose some $(k,\ell) \in \Arc(\Lambda)$ and let $z_\lambda = e_{k\ell} \in \fk l_\lambda$.   The coset $z_\lambda +\fk k_\lambda$ is independent of the choice of $(k,\ell)$
and
 spans the quotient $\fk l_\lambda / \fk k_\lambda$.  It follows that $z_\lambda + \fk k_\lambda$ and the cosets $e_{ij} + \fk k_\lambda$ for $(i,j) \in \Cr(\Lambda)$ provide a basis for $\fk s_\lambda / \fk k_\lambda$, and that the linear map defined by $z_\lambda + \fk k_\lambda \mapsto z_\Lambda \in \wt\fkcr_\Lambda(q)$ and $e_{ij} +\fk k_\lambda \mapsto e_{ij} \in \wt\fkcr_\Lambda(q)$ gives an algebra isomorphism $\fk s_\lambda /\fk k_\lambda \cong \wt\fkcr_\Lambda(q)$.

\end{enumerate}
On a technical note, the reader should observe that the isomorphism $\fk s_\lambda /\fk k_\lambda \cong \wt\fkcr_\Lambda(q)$ fails when $\Arc(\Lambda)= \varnothing$ and $\lambda = 0$, since then $\fk s_\lambda = \fk l_\lambda = \fk k_\lambda = \fkt_n(q)$ but $\wt\fkcr_\Lambda(q)\cong \FF_q$.

Applying Theorem \ref{count} to these constructions gives us a computable formula for $N_{\Lambda,e}(q)$.  Here, we write $\Irr(G;k)$ to denote the set of irreducible characters with degree $k$ of a group $G$.  

\begin{theorem}\label{ref} Fix a positive integer $n$, a prime power $q>1$, a nonnegative integer $e$, and a set partition $\Lambda \vdash[n]$. Then 
\[ N_{\Lambda,e}(q) =\frac{ \# \Irr\bigl(1+\wt\fkcr_\Lambda(q); q^f\bigr) -\#\Irr\bigl(1+\fkcr_\Lambda(q); q^f\bigr)}{q-1} ,\qquad\text{where } f = |\Cr(\Lambda)|-d(\Lambda)+e.\]
Furthermore, if all irreducible characters with degree $q^f$ of the algebra group $1+\wt\fkcr_\Lambda(q)$ are  Kirillov functions, then all irreducible constituents with degree $q^e$ of  supercharacters of $\UT_n(q)$ with shape $\Lambda$ are Kirillov functions.
\end{theorem}

\begin{proof}
When $\Cr(\Lambda) = \varnothing$ we have $\fkcr_\Lambda(q) = 0$ and $\wt\fkcr_\Lambda(q)\cong \FF_q$ so the given formula holds trivially, and all supercharacters with shape $\Lambda$ are Kirillov functions.  Therefore assume $\Arc(\Lambda)$ and $\Cr(\Lambda)$ are nonempty.   
Choose  $a \in \FF_q^\times$ and let $\lambda = a\cdot \sum_{(i,j) \in \Arc(\Lambda)}e_{ij}^* \in \fkt_n(q)^*$ so that we may view $\fkcr_\Lambda(q) = \fk s_\lambda / \fk l_\lambda$ and $\wt\fkcr_\Lambda(q) = \fk s_\lambda / \fk k_\lambda$.  If we identify $S_\lambda /K_\lambda$ with $1+\wt\fkcr_\Lambda(q)$, then  the quotient map $\pi \colon S_\lambda \to S_\lambda / K_\lambda $  in Theorem \ref{count} may be defined by
\[ \pi (1+X) = 1 + \sum_{(i,j) \in \Cr(\Lambda)} X_{ij} e_{ij} +\tfrac{\lambda(X)}{a} z_\Lambda \in 1+\wt\fkcr_\Lambda(q),\qquad\text{for }X \in \fk s_\lambda.\]
%
Thus $\pi(1+X) = 1+\tfrac{\lambda(X)}{a} z_\Lambda$ for all $X \in \fk l_\lambda$.   Since $q^f = \frac{\langle \chi_\lambda,\chi_\lambda\rangle}{\chi_\lambda(1)}  q^e$ by (\ref{fact}),
  it follows by Theorem \ref{count} that 
 $N_{\Lambda,e}(q)$ is  the number of irreducible characters $\psi$   of $1+\wt \fkcr_\Lambda(q)$ for which $\psi(1+t z_\Lambda)=q^f \cdot\theta(at)$ for all $t \in \FF_q$. 
 
Every irreducible character $\psi$ of $1+\wt\fkcr_\Lambda(q)$, however, has $\psi(1+tz_\Lambda) = \psi(1)\cdot \theta(bt)$ for all $t \in \FF_q$ for some (possibly zero) $b \in \FF_q$.  This is clear from the fact that $1+\FF_q \spanning\{z_\Lambda\}$ is a central algebra subgroup of $1 + \wt\fkcr_\Lambda(q)$ isomorphic to the additive group of $\FF_q$.  Since $a\in \FF_q^\times$ was arbitrary in the preceding paragraph, it follows that $(q-1)\cdot N_{\Lambda,e}(q)$ is the number of irreducible characters of $1 + \wt\fkcr_\Lambda(q)$ whose kernels do not contain $1+\FF_q \spanning\{z_\Lambda\}$.  Thus $(q-1)\cdot N_{\Lambda,e}(q)- \# \Irr(1 + \wt\fkcr_\Lambda(q); q^f)$ is the number of irreducible characters of degree $q^f$ of the quotient of $1+\wt\fkcr_\Lambda(q)$ by $1 + \FF_q\spanning\{z_\Lambda\}$.  This quotient is precisely $1+\fkcr_\Lambda(q)$, which completes the proof of the first part of the theorem.  The second part is a slightly weaker special case of the last part of Theorem \ref{count}.
\end{proof}

\begin{remark}
Let $\fk N$ be a finite-dimensional associative nilpotent $\ZZ$-algebra.  Evseev describes in \cite{E} an algorithm which attempts to compute polynomials in $q$ giving the number of irreducible characters of degree $q^e$ of the algebra group attached to the nilpotent $\FF_q$-algebra $\fk N \otimes_\ZZ \FF_q$.    The crossing algebras $\fkcr_\Lambda(q)$ and $\wt\fkcr_\Lambda(q)$ are certainly of this form.  Thus,
  on a purely theoretical level, the preceding result combined with Evseev's work gives an algorithm for computing $N_{\Lambda,e}(q)$ as a function in $q$.  More practically, Evseev has  implemented his algorithm in the computer algebra system {\sc{Magma}} \cite{Magma}, and this implementation \cite{implement, implement2} succeeds in computing polynomials in $q$ giving $\# \Irr\(1+\fkcr_\Lambda(q); q^f\)$ and 
  $\# \Irr(1+\wt\fkcr_\Lambda(q))$  in a large number of cases.    In this way, the preceding theorem allows us to undertake some of the more substantial computations promised in the introduction.   
  
  Besides counting, we also intend to show that all irreducible characters of $\UT_n(q)$ with a certain degree are Kirillov functions.
  Evseev's methods translate this problem into a tractable calculation in the following way.  
 As in \cite{E},  define an irreducible character of an algebra group  to be \emph{well-induced} if it is induced from a linear character $\tau$ of an algebra subgroup $1+\fk h$ with $\ker \tau \supset 1+\fk h^2$.  It is almost immediate from \cite[Theorem 5.5]{AndreAdjoint} that any well-induced irreducible character of an algebra group is a Kirillov function; we stated this fact as Proposition 4.1 in \cite{supp0}.
  Now, the algorithm in \cite{E} enumerates only well-induced characters, and thus when 
  it is successful in computing generic $q$-polynomials which count the irreducible characters of the algebra groups $1 + \fk N \otimes_\ZZ \FF_q$, 
 it follows that  all irreducible characters of these groups are Kirillov functions.
\end{remark}

\begin{example} \label{ex13}   
Suppose $\Lambda \vdash[13]$ is the set partition
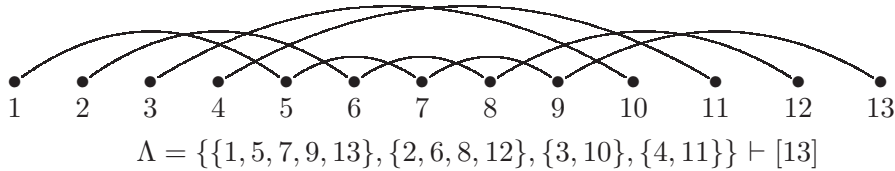
\[ \barr{c} \\ \\  \\[-8pt]
\xy<0.25cm,0.8cm> \xymatrix@R=-0.0cm@C=.5cm{
*{\bullet} \ar @/^1.6pc/ @{-} [rrrr] &
*{\bullet} \ar @/^1.6pc/ @{-} [rrrr] &
*{\bullet} \ar @/^2.4pc/ @{-} [rrrrrrr]  &
*{\bullet}  \ar @/^2.4pc/ @{-} [rrrrrrr]&
*{\bullet}  \ar @/^.8pc/ @{-} [rr]&
*{\bullet}  \ar @/^.8pc/ @{-} [rr] &
*{\bullet} \ar @/^.8pc/ @{-} [rr]&
*{\bullet} \ar @/^1.6pc/ @{-} [rrrr] &
*{\bullet} \ar @/^1.6pc/ @{-} [rrrr]  &
*{\bullet} &
*{\bullet} &
*{\bullet}&
*{\bullet} &
\\
1   & 
2 &
3  &
4 &
5 &
6 &
7 &
8 &
9 &
10 &
11 &
12 &
13 
}\endxy\\
\Lambda= \{ \{ 1,5,7,9,13\}, \{2,6,8,12\}, \{3,10\}, \{4,11\} \} \vdash[13]\\ \earr\]
$\UT_{13}(q)$ has $q(q-1)^{13}$ irreducible characters  which are not Kirillov functions by \cite[Theorem 1.4]{E} and \cite[Proposition 4.1]{supp0}, and  they all appear as constituents of supercharacters with  shape $\Lambda$ (see the remark following \cite[Proposition 3.2]{supp1}).  Hence, the original implementation of
 Evseev's algorithm should not be able to compute $N_{\Lambda,e}(q)$; however, the problems that arise in this special case are easily side-stepped.
 
In detail, Evseev's algorithm proceeds by recursively counting the characters of certain subgroups and quotients of the input, and it fails when the input is nontrivial yet cannot be reduced to an allowable subgroup or quotient.
For the crossing algebras $\fkcr_\Lambda(q)$ and $\wt\fkcr_\Lambda(q)$ with $\Lambda$ as above, this failure occurs when the algorithm is called recursively with an abelian algebra group as input.  The irreducible characters of such a group are easily counted even when they are not all well-induced (their number is the group's cardinality and their degrees are all one) and so after adding an appropriate if-then statement to Evseev's {\sc{Magma}} code \cite{implement}, as described in the comments in \cite{implement2}, we are able to compute via Theorem \ref{ref} that
\[N_{\Lambda,e}(q)= \begin{cases} 2(q-1)^4 + 7(q-1)^3 + 9(q-1)^2 + 5(q-1) + 1,&\text{if $e=15$},\\
3(q-1)^5 + 13(q-1)^4 + 22(q-1)^3 + 16(q-1)^2 + 4(q-1),&\text{if $e=16$}, \\
(q-1)^5 + 5(q-1)^4 + 7(q-1)^3 + 3(q-1)^2,&\text{if $e=17$}, \\
0,&\text{otherwise}.\end{cases}\]  Notably,  $N_{\Lambda,e}(q)$ is a polynomial in $q-1$ with nonnegative integer coefficients for all values of $e$.  
We mention that Example 3.1 in \cite{supp0} discusses how one can carry out a much less involved calculation to show that 
$N_{\Lambda}(2) = \sum_e N_{\Lambda,e}(2) = 98$, which is at least consistent with the more general formulas given here.

Of course, once we have modified Evseev's code in this way, it no longer holds that  if we can successfully compute $N_{\Lambda,e}(q)$ then  the irreducible constituents of a supercharacter with shape $\Lambda \vdash[n]$ are all well-induced.  Thus, it is important to stress that unless otherwise indicated, we do \emph{not} use this modified code in any of the subsequent calculations described in this work.

 \end{example}

The following corollary describes a common special case of Theorem \ref{ref}.  Say that a set $\cP$ of positions  above the diagonal in an $n\times n$ matrix is \emph{closed} if $(i,k) \in \cP$ whenever both $(i,j),(j,k) \in \cP$.  This is equivalent to the subspace 
\[ \fkt_{n,\cP}(q) \overset{\mathrm{def}} = \FF_q\spanning\{ e_{ij} : (i,j) \in \cP\} \subset \fkt_n(q)\] being a subalgebra.  We call a subalgebra of the form $\fkt_{n,\cP}(q)$ a \emph{pattern algebra} and the corresponding algebra group $\UT_{n,\cP}(q) \overset{\mathrm{def}} = 1 + \fkt_{n,\cP}(q)$ a \emph{pattern group}. 

\begin{corollary}\label{ref-cor} Retain the notation of Theorem \ref{ref}. 
 If for all $i,j,k,l,m \in [n]$ at most one of $(i,j,k,l)$ or $(j,k,l,m)$ is a crossing of $\Lambda$, then 
 \[N_{\Lambda,e}(q)=\# \Irr\(1+\fkcr_\Lambda(q);q^f\),\qquad\text{where }f =|\Cr(\Lambda)| - d(\Lambda)+e.\]  Furthermore, if this holds and all irreducible characters of the algebra group $1+\fkcr_\Lambda(q)$ are  Kirillov functions, then all irreducible constituents of  supercharacters of $\UT_n(q)$ with shape $\Lambda$ are Kirillov functions.
The given condition holds in particular when $\Cr(\Lambda)$ is closed, in which case $1+\fkcr_\Lambda(q)$ is isomorphic to the pattern group $\UT_{n,\Cr(\Lambda)}(q)$.

\end{corollary}

\begin{proof}
By construction $\Arc(\Lambda) \cap \Cr(\Lambda) = \varnothing$, and if our condition obtains, then $(i,j),(j,k) \in \Cr(\Lambda)$ implies $(i,k) \notin \Arc(\Lambda)$.  It follows in this case that $1+\wt\fkcr_\Lambda(q)$ is the internal direct product of $1+\FF_q\spanning\{z_\Lambda\}\cong \FF_q^+$ and a subgroup isomorphic to $1+\fkcr_\Lambda(q)$, so in particular $\#\Irr(1+\wt\fkcr_\Lambda(q);q^f) = q \cdot \#\Irr(1+\fkcr_\Lambda(q);q^f)$.  All irreducible characters of the abelian algebra group $1+\FF_q\spanning\{z_\Lambda\}$ are Kirillov functions, whence it follows that the same is true of all irreducible characters of $1+\wt\fkcr_\Lambda(q)$ if and only if every irreducible character of $1+\fkcr_\Lambda(q)$ is a Kirillov function.
The first half of the corollary now follows from the preceding theorem.  The last part is a consequence of the fact that $\fkcr_\Lambda(q)$ is indeed equal to the pattern algebra $\fkt_{n,\Cr(\Lambda)}(q)$ if $\Cr(\Lambda)$ is closed.
\end{proof}

In view of this corollary, it is worth noting that the natural analogue of Lehrer's conjecture fails for certain pattern groups.   Indeed, Halasi \cite{Hthesis} has recently shown (non-constructively)   that for some sufficiently large $n$ there exists a closed set of upper triangular positions $\cP$ such that
\begin{enumerate}
\item[(a)] $X^3 = 0$ for all $X \in \fkt_{n,\cP}(q)$; 
\item[(b)] The number of irreducible characters of the pattern group $\UT_{n,\cP}(q)$ is not a polynomial function in $q$, and in fact cannot be described by any finite set of polynomials in $q$  \cite[Theorem 4.9]{Hthesis}.
\end{enumerate}
\begin{remark}
In this situation, part (b) is true not only for the number of conjugacy classes / irreducible characters of $\UT_{n,\cP}(q)$, but also for the number of its superclasses / supercharacters, since (a) implies that $\cP$ has no 4-chains whence every supercharacter is irreducible by \cite[Proposition 5.1]{DT}.
\end{remark}
  Thus, if one could find $\Lambda \vdash[n]$ so that $\Cr(\Lambda)$ is an arbitrary closed set of positions, or at least a pattern $\cP$ for which (b) holds, then
the preceding corollary with Halasi's result would immediately disprove Conjecture \ref{ours}.

One cannot immediately apply this direct method of disproof, as the patterns which occur as $\Cr(\Lambda)$ for $\Lambda \vdash[n]$ are not arbitrary.  One can show, for example, that if $\Cr(\Lambda)$ is closed then $\UT_{n,\Cr(\Lambda)}(q)$ is never isomorphic to the commutator subgroup of $\UT_k(q)$ for $k\geq 5$.
In the next proposition we describe how to construct one obvious family of pattern groups whose conjugacy classes are counted by $N_\Lambda(q)$.  In general, however, the question of precisely which closed sets of positions may occur as $\Cr(\Lambda)$ for $\Lambda\vdash[n]$$-$and whether Halasi's methods can be adapted to disprove Conjecture \ref{ours}$-$remains open.  

\def\J{\cJ}

\begin{proposition} Fix a positive integer $n$ and let $\cJ = \{ (i,j) : 1\leq i <j \leq n \}$.  If $\cP\subset \cJ$ has the property that both $\cP$ and $\cJ\setminus \cP$ are closed, then there exists $\Lambda\vdash[2n]$ such that $\Cr(\Lambda) = \cP$.  
\end{proposition}

\begin{proof}

 $\cP\subset \J$ satisfies our hypothesis if and only if the relation $\prec$ on $[n]$, given by setting  $i \prec j$ whenever $ (i,j) \in \cP$ or $(j,i) \in \J\setminus \cP$, is a total order.    
 Let  $h_\cP \colon [n] \to \NN$ be the height function of this total order, 
%
%
and let $\Lambda \vdash[2n]$ be the set partition with arc set $(j, n+h_\cP(j))$ for $j\in [n]$.  This is well-defined since $h_\cP \colon [n] \to [n]$ is a permutation, and one obtains
$\Cr(\Lambda) = \{ (i,j) : 1\leq i <j \leq n\text{ and } h_\cP(i) < h_\cP(j)\}$ by definition. This set is precisely $\cP$, since $h_\cP(i) < h_\cP(j)$ if and only if $i \prec j$, and when $i<j$ then this is equivalent to $(i,j) \in \cP$.
\end{proof}
%
%
%

\begin{example}
If $\Lambda= \{ \{1,n+1\}, \{2,n+2\},\dots, \{n,2n\}\} \vdash[2n]$ then  $N_{\Lambda,e}(q)$ is the number of irreducible characters of $\UT_n(q)$ of degree $q^{f}$ where $f = e- n(n-1)/2$.
\end{example}

We mention also that the supercharacters of the normal pattern subgroups $\UT_{n,\cP}(q) \vartriangleleft \UT_n(q)$ have been classified and possess a relatively explicit indexing set analogous to the set of quasi-monomial maps in $\fkt_n(q)^*$; see \cite{M2}.  
It may be possible to define a ``shape'' for these supercharacters, given  by some mild generalization of a set partition, which is similarly invariant under the action of an appropriate subgroup of $\Aut\( \UT_{n,\cP}(q)\)$.  This would presumably allow one to define and compute analogues of $N_{\Lambda,e}(q)$ for $\UT_{n,\cP}(q)$ using Evseev's algorithm with Theorem \ref{count}.  By extending the techniques described in the next sections, one might discover a version of Theorem \ref{intro} for, say, the commutator subgroups of $\UT_n(q)$ or some other family of normal pattern subgroups.

\subsection{Connectedness for set partitions}

\def\split{\mathrm{Split}}

While in principle we can use the results of the previous section and \cite{E} to compute $N_{\Lambda,e}(q)$ for all set partitions $\Lambda \vdash[n]$, this quickly grows to an enormous calculation.  The complexity of  this undertaking is significantly diminished by a useful factorization of 
$N_{\Lambda,e}(q)$, which we describe here.  The factors will correspond to the components of $\Lambda$ which are connected in a certain strong sense.  Leading up to our precise statement, we first describe  three increasingly restrictive notions of connectedness for set partitions.

The first notion is that of an atomic set partition, the definition of which we take from \cite{BZ}. Given two set partitions $\Gamma \vdash[m]$ and $\Lambda \vdash[n]$, define $\Gamma \bar \Lambda = \Gamma \cup(\Lambda+m) \vdash[m+n]$, where $\Lambda+m$ is the set partition of $\{m+1,m+2,\dots,m+n\}$ formed by adding $m$ to the entries in each part of $\Lambda$.  A set partition $\Lambda \vdash[n]$ is \emph{splittable} if there exist set partitions $A$, $B$ with 
 $\Lambda = A\bar B$, and \emph{atomic} otherwise. The \emph{split} of a set partition $\Lambda \vdash[n]$ is then the unique sequence  
 \[ \split(\Lambda)= \( \Lambda^{(1)}, \Lambda^{(2)}, \dots,\Lambda^{(d)}\)\] such that $\Lambda^{(i)}$ is an atomic set partition and $\Lambda = \Lambda^{(1)} \bar \Lambda^{(2)}\bar \cdots \bar \Lambda^{(d)}$.  
  Bergeron and Zabrocki show in \cite{BZ} that atomic set partitions index a free generating set of the Hopf algebra ${NCSym}$ of symmetric functions in noncommuting variables.  In fact, there is a natural way of identifying $NCSym$ with the space of superclass functions on $\UT_n(2)$, a fascinating connection explored  in  \cite{alia}.

Our second notion is that of a connected set partition.  
We say that a set partition $\Lambda\vdash[n]$ is \emph{disconnected} if the union of a subset of its parts is a proper, nonempty subinterval of $[n]$. 
Equivalently and more generally (as a consequence of \cite[Lemma 2.5]{Klazar}, for example),  a set partition $\Lambda\vdash\cS \subset \NN$ is \emph{disconnected} if and only if there exists a nonempty, proper subset $\Gamma \subset \Lambda$ such that 
\be\label{disconnected} \Cr(\Lambda) = \Cr(\Gamma) \cup \Cr(\Lambda\setminus\Gamma).\ee  
Note that this is well-defined as any subset of $\Lambda$ is a set partition of a subset of $\cS$.
  Naturally, $\Lambda$ is \emph{connected} if not disconnected.  If $\Gamma\subset\Lambda$ is nonempty and connected and equation (\ref{disconnected}) holds, then we say that $\Gamma$ is a \emph{connected component} of $\Lambda$.  A set partition $\Lambda$ then has a well-defined set of connected components, which we denote by $\comp(\Lambda)$.
  Bender, Odlyzko, and Richmond study the asymptotic number of connected set partitions in \cite{Bend1,Bend2}, where they are called irreducible. 
More recently,  Klazar describes a generating function and a recurrence for their enumeration in \cite{Klazar}.

  Our final notion is apparently the least standard.  We say that a set partition $\Lambda$ of a set $ \cS \subset \NN$ is \emph{crossing-connected} if $\Lambda$ is connected and $\Arc(\Lambda)$ has at most one equivalence class with respect to the equivalence relation $\sim$ generated by setting 
   \[ (i,k) \sim (j,l)\quad\text{whenever $(i,j,k,l)$ is a crossing of $\Lambda$}.\]
We note that $\Arc(\Lambda)$ has zero equivalence classes with respect to $\sim$ if and only if $\Arc(\Lambda)=\varnothing$, and in this case $\Lambda$ is crossing-connected if and only if $\Lambda$ partitions a set with one element.
  The \emph{crossing-connected components} of a set partition $\Lambda\vdash\cS$ are the crossing-connected set partitions $\Gamma$ such that either
  \begin{enumerate}
  \item[(1)] $\Gamma = \{ \{i\}\}$ where $\{i\}$ is a singleton part of $\Lambda$.
  \item[(2)] $\Arc(\Gamma)$ is an equivalence class of $\Arc(\Lambda)$ with respect to $\sim$.
  \end{enumerate}
  We denote the set of crossing-connected components of $\Lambda$ by $\crcomp(\Lambda)$.
  Unlike connected components, a crossing-connected component $\Gamma$ of $\Lambda$ may not have $\Gamma \subset \Lambda$; however, one always has $\Arc(\Lambda) = \bigcup_{\Gamma \in \crcomp(\Lambda)} \Arc(\Gamma)$ and $\Cr(\Lambda)=\bigcup_{\Gamma \in \crcomp(\Lambda)} \Cr(\Gamma)$ where the unions are disjoint, since two crossing arcs belong to same equivalence class.    

A connected set partition is atomic, and a crossing-connected set partition is connected.  Intuitively,  consider the standard representation of $\Lambda$ drawn in the plane with all vertices collinear and all edges on the same side of the determined line.  Then $\Lambda$ is connected if and only if one can travel between any two vertices by moving along arcs, where one can switch from one arc to another at a crossing or at a vertex.  In the same setup, $\Lambda$ is crossing-connected if and only if  the same feat is possible with the added condition that one can switch between arcs only at crossings.  For example, consider the following set partitions $A,B,C\vdash[5]$:
  \[
   \barr{c} \\ \\[-10pt] 
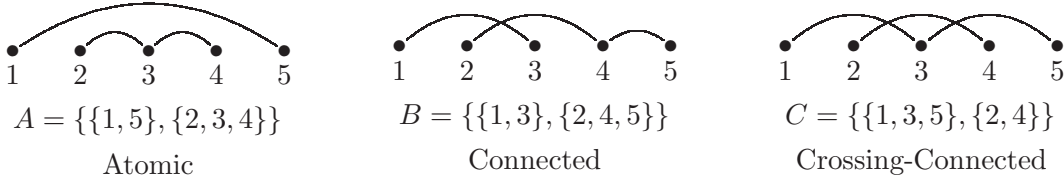
\xy<0.25cm,0.8cm> \xymatrix@R=-0.0cm@C=.5cm{
*{\bullet} \ar @/^1.5pc/ @{-} [rrrr]   & 
*{\bullet} \ar @/^0.6pc/ @{-} [r] &
*{\bullet}  \ar @/^0.6pc/ @{-} [r]   &
*{\bullet} &
*{\bullet} \\
1   & 
2 &
3  &
4 &
5 
}\endxy\\
A= \{ \{ 1,5\} ,\{2,3,4\} \} \\ \\[-10pt]
  \text{Atomic}
  \earr
\qquad
 \barr{c} \\ \xy<0.25cm,0.8cm> \xymatrix@R=-0.0cm@C=.5cm{
*{\bullet} \ar @/^1pc/ @{-} [rr]   & 
*{\bullet} \ar @/^1pc/ @{-} [rr] &
*{\bullet}  &
*{\bullet} \ar @/^0.5pc/ @{-} [r]&
*{\bullet} \\
1   & 
2 &
3  &
4 &
5 
}\endxy\\
B=  \{ \{ 1,3\} ,\{2,4,5\} \}\\ \\[-10pt]
  \text{Connected} \\
  \earr
  \qquad
 \barr{c} \\ \xy<0.25cm,0.8cm> \xymatrix@R=-0.0cm@C=.5cm{
*{\bullet} \ar @/^1pc/ @{-} [rr]   & 
*{\bullet} \ar @/^1pc/ @{-} [rr] &
*{\bullet} \ar @/^1pc/ @{-} [rr] &
*{\bullet}  &
*{\bullet} \\
1   & 
2 &
3  &
4 &
5 
}\endxy\\
C=  \{ \{ 1,3,5\} ,\{2,4\} \} \\ \\[-10pt]
  \text{Crossing-Connected} \\
  \earr
   \]
   The first set partition $A$ is atomic but not connected: its two connected components are $\{ \{1,5\} \}$, $\{ \{2,3,4\}\}$ and its three crossing-connected components are $\{ \{1,5\} \}$, $\{ \{2,3\}\}$, $\{\{3,4\}\}$. 
   Similarly, $B$ is connected but not crossing-connected: its two crossing-connected components are $\{ \{1,3\}, \{2,4\} \}$ and $\{ \{4,5\}\}$.  The third set partition $C$ is crossing-connected (in fact, $C$ is the only crossing-connected set partition of $[5]$), and therefore connected and atomic.

 We list the numbers of these various types of set partitions in Table \ref{tbl1}. 
Here we let $B_n$ denote the Bell numbers counting giving the number of  set partitions of $[n]$.  The modified numbers $B_n^{\text{type}}$ are self-explanatory; recurrence and asymptotic formulas for $B_n^{\text{crossing-connected}}$ are desired.

\begin{table}[h]
\[
 \barr{| c | l l l l |}
\hline 
n 
& B_n
  & B_n^{\mathrm{atomic}}
& B _n^{\mathrm{connected}}
& B_n^{\mathrm{crossing-connected}} 
\\
\hline
1 & 1 & 1 & 1& 1  \\
2 & 2 & 1 & 1& 1 \\
3 & 5 & 2 & 1& 0 \\
4 & 15 & 6 & 2& 1 \\
5 & 52 & 22 & 6& 1 \\
6 & 203 & 92 & 21& 5 \\
7 & 877 & 426 & 85& 16 \\
8 & 4,140 & 2,146 & 385& 69 \\
9 & 21,147 & 11,624 & 1,907 & 316 \\
10 & 115,975 & 67,146 & 10,205 & 1,591 \\
11 & 678,570 & 411,142 & 58,455 & 8,614 \\
12 & 4,213,597 & 2,656,052 & 355,884& 49,841 \\
13 & 27,644,437 & 18,035,178 & 2,290,536& 306,043 \\
14& 190,899,322 & 128,318,314 & 15,518,391&1,984,493 \\
15 & 1,382,958,545 & 954,086,192 & 110,283,179& 13,533,898
\\
\hline
\text{OEIS \cite{OEIS}:} & \text{\htmladdnormallink{A000110} {http://oeis.org/A000110}}
 &  \text{\htmladdnormallink{A074664} {http://oeis.org/A074664}}
 & \text{\htmladdnormallink{A099947} {http://oeis.org/A099947}}&\text{N/A}
\\\hline
\earr
\]
\caption{Counting atomic, connected, and crossing-connected set partitions of $[n]$}
\label{tbl1}
\end{table}

%
%

\subsection{Factorizations of $N_\Lambda(q)$ and $N_{\Lambda,e}(q)$}

The various components of $\Lambda\vdash[n]$ just defined are not necessarily partitions of sets of consecutive integers, and so to write down a decomposition of $N_{\Lambda,e}(q)$ we must explain what this notation means for an arbitrary set partition.  To this end, we observe that if $\Lambda \vdash\cS\subset \NN$ and $k=|\cS|$ then there is a unique ordering-preserving bijection $\cS\to  [k]$.  Following the convention of \cite{Groth}, we call the set partition of $[k]$ given by applying this bijection to the parts of $\Lambda$ the \emph{standardization} of $\Lambda$ and denote it $\st(\Lambda)$. For example, 
\[ \Lambda = \{ \{ 4,9\} ,\{6,14\}, \{10\}\}\qquad\text{has}\qquad \st(\Lambda) = \{ \{1,3\},\{2,5\},\{4\}\} \vdash[5].\] 
Observe that the crossing sets of $\Lambda$ and $\st(\Lambda)$ have the same cardinality but $d(\Lambda) \geq d(\st(\Lambda))$.
 
For a set partition $\Lambda\vdash\cS$ of an arbitrary finite subset $\cS\subset \NN$, we now define
\be\label{arb-def} 
\ba N_{\Lambda}(q) &\overset{\mathrm{def}}= N_{\st(\Lambda)}(q), \\
 N_{\Lambda,e}(q) &\overset{\mathrm{def}}= N_{\st(\Lambda), e-f_\Lambda}(q),\qquad\text{where }f_\Lambda= d(\Lambda)-d(\st(\Lambda)).\ea\ee   
We may now state this section's main theorem.  Here we recall that a weak composition of a nonnegative integer $k$ is a sequence of nonnegative integers whose sum is $k$.

\begin{theorem}\label{products}
For any positive integer $n$, prime power $q>1$, nonnegative integer $e$, and set partition $\Lambda \vdash[n]$, we have 
\[\displaystyle N_{\Lambda,e}(q) = \sum_{\textbf{w}} \prod_{\Gamma \in \crcomp(\Lambda)} N_{\Gamma, \textbf{w}_\Gamma}(q),\] where the sum is over all weak compositions $\textbf{w} = (\textbf{w}_\Gamma)$ of $e$ with $|\crcomp(\Lambda)|$ parts.
Furthermore
\[\displaystyle N_\Lambda(q) = \prod_{\Gamma \in \split(\Lambda)} N_\Gamma(q) = \prod_{\Gamma \in \comp(\Lambda)} N_\Gamma(q) = \prod_{\Gamma \in \crcomp(\Lambda)} N_\Gamma(q).\] 
\end{theorem}

Immediately, we have this corollary:

\begin{corollary} Conjecture \ref{ours} holds if and only if it holds for crossing-connected set partitions.  That is, $N_{\Lambda,e}(q)$ is a polynomial in $q$ with integer coefficients for all $\Lambda\vdash[n]$ and $e\in \ZZ$ if and only if the same is true of $N_{\Gamma,f}(q)$  for all crossing-connected set partitions $\Gamma \vdash[k]$ with $k\leq n$ and $f \in \ZZ$.
\end{corollary}

Our proof of the theorem will follow from  two short lemmas, which we state below in rapid succession.  
The first of these is an immediate consequence of \cite[Lemma 3.4]{Le}; we provide a short proof using Lemmas   \ref{obs-las} and \ref{form} for completeness.

\begin{lemma}\label{lem1}
Suppose $\Lambda, A,B \vdash[n]$ are such that $\Arc(\Lambda)$ is the disjoint union of $\Arc(A)$ and $\Arc(B)$ and $\Cr(\Lambda)$ is the disjoint union of $\Cr(A)$ and $\Cr(B)$.  Then 
\[ N_{\Lambda,e}(q) = \sum_{a+b=e} N_{A,a}(q) \cdot N_{B,b}(q).\]
\end{lemma}

\begin{proof}
Let $\alpha,\beta \in \fkt_n(q)^*$ be quasi-monomial with shapes $A,B$, respectively.  As $\Arc(A)$ and $\Arc(B)$ are disjoint and their union is $\Arc(\Lambda)$, it follows that $\lambda =\alpha+\beta$ is quasi-monomial with shape $\Lambda$ and $\chi_{\lambda} = \chi_\alpha \otimes \chi_\beta$ by Lemma \ref{form}.   Noting  (\ref{fact}), our claim follows by Lemma \ref{obs-las}.
\end{proof}

\begin{lemma}\label{lem2} Suppose $A \vdash[n]$ and let $\Gamma$ be the set partition formed by removing from $A$ all of its singleton parts.  Then $N_{A,e}(q) = N_{\Gamma,e}(q)$.
\end{lemma}

Note that in this statement $\Gamma$ is not necessarily a set partition of $[n]$, and so the normalization in (\ref{arb-def}) becomes important.

\def\Y{\mathcal{Y}}

\begin{proof}
Suppose $\Gamma \vdash \cS = \{ i_1< i_2 < \dots < i_k\}$.  Let $B\vdash[k]$ denote the standardization of $\Gamma$, and observe that the reindexing map $j \mapsto i_j$ induces obvious isomorphisms $\fkcr_A(q) \cong \fkcr_B(q)$ and $\wt\fkcr_A(q)\cong \wt\fkcr_B(q)$.  By definition  $N_{\Gamma,e}(q) = N_{B,e'}$ where $e'=e-d(\Gamma)+d(B)$, and $|\Cr(A)| = |\Cr(B)|$ and $d(A) = d(\Gamma)$.  Thus $|\Cr(A)|- d(A) +e = |\Cr(B)| - d(B) +e'$, and so the claim $N_{\Gamma,e}(q)=N_{A,e}(q)$ is apparent from Theorem \ref{ref}.  
\end{proof}

\begin{proof}[Proof of Theorem \ref{products}]
Let $\Lambda \vdash[n]$, and for each $\Gamma \in \crcomp(\Lambda)$ form $\wt\Gamma \vdash[n]$ by adding to $\Gamma$ a sequence of singleton parts, so that $\Arc(\Gamma) = \Arc(\wt\Gamma)$. It follows by inductively applying Lemma \ref{lem1} that 
$N_{\Lambda,e}(q) = \sum_{\textbf{w}} \prod_{\Gamma \in \crcomp(\Lambda)} N_{\wt\Gamma, \textbf{w}_\Gamma}(q)$.  By the preceding lemma  $N_{\wt\Gamma,e}(q) = N_{\Gamma,e}(q)$, which brings our equation into the desired form.  The formulas for $N_\Lambda(q)$ now follow by summing the formula for $N_{\Lambda,e}(q)$ over all $e \geq 0$.  
\end{proof}

\begin{example} If $\Lambda \vdash[15]$ is the set partition
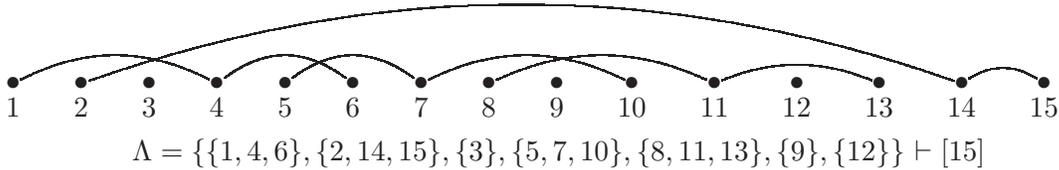
\[ \barr{c} \\ \\  \\[-8pt]
\xy<0.25cm,0.8cm> \xymatrix@R=-0.0cm@C=.5cm{
*{\bullet} \ar @/^.9pc/ @{-} [rrr] &
*{\bullet} \ar @/^2.5pc/ @{-} [rrrrrrrrrrrr] &
*{\bullet} &
*{\bullet}  \ar @/^.9pc/ @{-} [rr]&
*{\bullet}  \ar @/^.9pc/ @{-} [rr]&
*{\bullet} &
*{\bullet} \ar @/^.9pc/ @{-} [rrr]&
*{\bullet} \ar @/^.9pc/ @{-} [rrr] &
*{\bullet}  &
*{\bullet} &
*{\bullet} \ar @/^.6pc/ @{-} [rr]&
*{\bullet}&
*{\bullet} &
*{\bullet} \ar @/^.5pc/ @{-} [r]&
*{\bullet}
\\
1   & 
2 &
3  &
4 &
5 &
6 &
7 &
8 &
9 &
10 &
11 &
12 &
13 &
14 &
15 &
}\endxy\\
\Lambda=  \{ \{1,4,6\}, \{2,14,15\}, \{3\}, \{5,7,10\},  \{8,11,13\}, \{9\}, \{12\} \} \vdash[15]\\ \earr\]
then $\Lambda$ has eight crossing-connected components and $N_\Lambda(q) =  N_A(q)^3\cdot N_B(q)^2 \cdot N_C(q)^3$ where $A = \{ \{ 1,3\},\{2,4\} \} \vdash[4]$ and  $B=\{ \{1,2\}\} \vdash[2]$ and $C =\{ \{1\}\} \vdash[1]$.  Using Corollary \ref{ref-cor}, one finds $N_A(q) = N_{A,1}(q)= q$ and $N_B(q) = N_{B,0}(q)= N_C(q) = N_{C,0}(q)= 1$, which allows us to compute $N_\Lambda(q) = N_{\Lambda,17}(q) = q^3$.
\end{example}
  

We see from Table 1 that Theorem \ref{products} reduces the amount of work required to compute $N_{\Lambda,e}(q)$ for all $\Lambda \vdash[n]$ quite significantly: the fraction of set partitions of $[15]$ which are crossing-connected is a little less than $1/100$.
 Moreover, we can verify Conjecture \ref{ours} for small values of $n$  by inspection.  Aiding us in this is the following  proposition.

\begin{proposition}\label{eval}
Fix a positive integer $n$, a prime power $q>1$, and set partition $\Lambda \vdash[n]$.
\begin{enumerate}
\item[(1)] If $|\Cr(\Lambda)| =t \in \{0,1\}$ then $N_{\Lambda,e}(q) =\begin{cases} q^t,&\text{if }e=d(\Lambda)-t, \\ 0,&\text{otherwise}.\end{cases}$
\item[(2)] If $|\Cr(\Lambda)|=2$ then $N_{\Lambda,e}(q) =\begin{cases} q^{2t-2},&\text{if }e=d(\Lambda)-t, \\ 0,&\text{otherwise},\end{cases}$  where 
\[ t = \begin{cases}  1,&\text{if there are $i,j,k$ with $(i,j),(j,k) \in \Cr(\Lambda)$ and $(i,k) \in \Arc(\Lambda)$},\\ 2,&\text{otherwise}.\end{cases}\] 

\item[(3)] $N_{\Lambda,e}(q) = N_{\Lambda^\dag,e}(q)$, where $\Lambda^\dag\vdash[n]$ is   given by applying the map $i\mapsto n+1-i$ to the parts of $\Lambda$.  Also, all irreducible constituents of supercharacters with shape $\Lambda$ are Kirillov functions if and only if  all irreducible constituents of supercharacters with shape $\Lambda^\dag$ are Kirillov functions.
\end{enumerate}
In cases (1) and (2), all irreducible constituents of supercharacters with shape $\Lambda$ are Kirillov functions. 
\end{proposition}

\begin{proof}
If $|\Cr(\Lambda)|=0$ then any supercharacter with shape $\Lambda$ is irreducible by (\ref{fact}) and therefore a Kirillov function.  If $|\Cr(\Lambda)|=1$ then $\Cr(\Lambda)$ is a closed set of positions corresponding to a pattern group isomorphic to the additive group of $\FF_q$, all of whose irreducible characters are Kirillov functions.  Our formula in (1) thus follows from 
 Corollary \ref{ref-cor}.

If $t=2$ in (2) then the condition in  Corollary \ref{ref-cor} holds.  In this case the algebra group $1+\fkcr_\Lambda(q)$ is isomorphic to the direct product of two copies of the additive group $\FF_q^+$, and one sees that all of its irreducible characters are Kirillov functions.  Therefore any supercharacter with shape $\Lambda$ is equal to the multiplicity-free sum of $q^2$ irreducible characters of the same degree, which are all Kirillov functions.
On the other hand, the case $t=1$ in (2) follows by Theorem \ref{max-cross}.

If $X$ is an $n\times n$ matrix then let $X^\dag$ denote its \emph{backwards transpose}: this is the $n\times n$ matrix with $(X^\dag)_{i,j} = X_{n+1-j, n+1-i}$.  Any supercharacter with shape $\Lambda^\dag$ is then given by composing a supercharacter with shape $\Lambda$ with the automorphism $\varphi : g\mapsto (g^{-1})^\dag$ of $\UT_n(q)$.  Composition with an automorphism permutes the set of all irreducible characters of a given degree, proving the first statement in part  (3). To prove the second statement, observe that if a Kirillov function $\psi_\lambda$ for $\lambda \in \fkt_n(q)^*$ is a character of $\UT_n(q)$, then $\psi_\lambda \circ \varphi$ is  also a character, and so we have $\psi_\lambda \circ\varphi(g) = \overline{\psi_\lambda(g^\dag)}$ for $g\in \UT_n(q)$, which implies that $\psi_\lambda \circ\varphi$ is the Kirillov function indexed by the functional $X \mapsto \lambda(-X^\dag)$ in $\fkt_n(q)^*$.
Since $\varphi$ is an involution, we see by symmetry that an irreducible character $\psi$ of $ \UT_n(q)$ is a Kirillov function if and only if the irreducible character $\psi\circ \varphi$ is a Kirillov function. This suffices to prove our last assertion in (3) since composition with $\varphi$ exchanges the sets of irreducible constituents of supercharacters with shapes $\Lambda $ and $\Lambda^\dag$.
\end{proof}

Every crossing-connected set partition $\Lambda \vdash[n]$ for $n\leq 6$ has $|\Cr(\Lambda)| \leq 2$ except two which have $|\Cr(\Lambda)| = 3$.  The exceptions are
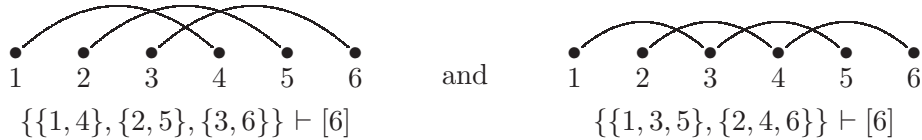
\[    \barr{c} \\ \xy<0.25cm,0.8cm> \xymatrix@R=-0.0cm@C=.5cm{
*{\bullet} \ar @/^1.5pc/ @{-} [rrr]   & 
*{\bullet} \ar @/^1.5pc/ @{-} [rrr] &
*{\bullet} \ar @/^1.5pc/ @{-} [rrr] &
*{\bullet} &
*{\bullet} &
*{\bullet} \\
1   & 
2 &
3  &
4 &
5 &
6
}\endxy\\
  \{ \{ 1,4\} ,\{2,5\}, \{3,6\} \} \vdash[6]
  \earr
  \qquad\text{and}\qquad 
   \barr{c} \\ \xy<0.25cm,0.8cm> \xymatrix@R=-0.0cm@C=.5cm{
*{\bullet} \ar @/^1pc/ @{-} [rr]   & 
*{\bullet} \ar @/^1pc/ @{-} [rr] &
*{\bullet} \ar @/^1pc/ @{-} [rr] &
*{\bullet}  \ar @/^1pc/ @{-} [rr]&
*{\bullet} &
*{\bullet} \\
1   & 
2 &
3  &
4 &
5 &
6
}\endxy\\
  \{ \{ 1,3,5\} ,\{2,4,6\} \} \vdash[6]
 \earr\] and so Conjecture \ref{ours} holds for $n\leq 5$.  If $\Lambda = \{ \{ 1,4\} ,\{2,5\}, \{3,6\} \} \vdash[6]$ then  $\Cr(\Lambda)$ is closed and $\UT_{6,\Cr(\Lambda)}(q) \cong \UT_3(q)$, so $N_{\Lambda,e}(q)$ is a polynomial in $q$ for all $e$ by Corollary \ref{ref-cor}.  
To treat the second case, we note that $\Lambda= \{ \{ 1,3,5\} ,\{2,4,6\} \} \vdash[6]$ is the shape of the supercharacter $\chi_\lambda$ indexed by 
\[\lambda = e_{13}^* + e_{24}^* + e_{35}^* + e_{46}^* \in \fkt_6(q)^*.\]  
Example 2.1 in \cite{supp1} computes the irreducible constituents of this supercharacter: $\chi_\lambda$ is a sum of $q$ distinct irreducible characters of degree $q^2$, each appearing with multiplicity $q$.  Hence $N_{\Lambda,e}(q)=q$ if $e=2$ and zero otherwise. By the corollary to Theorem \ref{products} we conclude:

 
\begin{observation} Conjecture \ref{ours} holds for $n\leq 6$.  
\end{observation}

To check Conjecture \ref{ours} for higher values of $n$, we must apply Evseev's algorithm  to the crossing algebra groups $1+\fkcr_\Lambda(q)$ and $1+\wt\fkcr_\Lambda(q)$ as outlined 
in the remark following Theorem \ref{ref}.   
For $n\leq 11$, Evseev's {\sc{Magma}} implementation (without the modification described in Example \ref{ex13}) 
succeeds in computing polynomial formulas for $N_{\Lambda,e}(q)$ for every crossing-connected set partition $\Lambda \vdash[n]$.   Evseev's algorithm fails to compute $N_{\Lambda,e}(q)$ for exactly one set partition $\Lambda \vdash[12]$, but the algorithm does succeed for the transposed set partition $\Lambda^\dag$. 
%
%
In light of Proposition \ref{eval} and \cite[Proposition 4.1]{supp0}, this suffices to prove the first half of Theorem 1.4 in \cite{E}, which states that all irreducible characters of $\UT_n(q)$ are well-induced for $n\leq 12$.  

Evseev's algorithm similarly fails to compute $N_{\Lambda,e}(q)$ for exactly 34 crossing-connected set partitions $\Lambda \vdash[13]$.   In 32 of these problem cases, the algorithm  succeeds for the transposed partition $\Lambda^\dag$, so we can again invoke Proposition \ref{eval}.  The two remaining set partitions $\Lambda$ both have $\Lambda^\dag =\Lambda$, but
one of these is the set partition discussed in Example \ref{ex13}.
The other remaining case is $\Lambda = \{\{1,5,9,13\},\{2,8\},\{3,7,11\},\{4,10\},\{6,12\}\}$, and we can determine $N_{\Lambda,e}(q)$ for this set partition by computing the sum (\ref{imply}) over all successful cases and then subtracting this from Evseev's formulas for $N_{13,e}(q)$.  

The {\sc{Magma}} code we used to carry out these  computations is available online \cite{implement,implement2}.
The following theorem summarizes our results.

\begin{theorem}
Conjecture \ref{ours} holds for $n\leq 13$.  In particular, if $\Lambda \vdash[n]$ and $e \in \ZZ$, then
\begin{enumerate}
\item[(1)] $N_{\Lambda,e}(q)$ is a polynomial in $q-1$ with nonnegative integer coefficients if $n\leq 12$.

\item[(2)] If $\Lambda = \{ \{ 1,6,8,13\}, \{2,7,12\}, \{3,9\}, \{4,10\}, \{5,11\} \} \vdash[13]$ then $N_{\Lambda,20}(q)$ is a polynomial in $q-1$ with both positive and negative integer coefficients.  In detail,
\[N_{\Lambda,e}(q) = \begin{cases} (q-1)^4 + 4(q-1)^3 + 6(q-1)^2 + 4(q-1)+1,&\text{if }e=18, \\
2(q-1)^5 + 11(q-1)^4 + 22(q-1)^3 + 19(q-1)^2 + 6(q-1),&\text{if }e=19, \\
3(q-1)^4 + 7(q-1)^3 + 3(q-1)^2-(q-1),&\text{if }e=20, \\
(q-1)^3,&\text{if }e=21, \\
0,&\text{otherwise}.\end{cases}\]
For all other set partitions $\Lambda \vdash[13]$, $N_{\Lambda,e}(q)$ is a polynomial  in $q-1$ with nonnegative integer coefficients for all values of $e$.
\end{enumerate}
\end{theorem}

\begin{remarks}
\begin{enumerate}
\item[]
\item[(i)]
 The formulas we get for $N_{n,e}(q) = \sum_{\Lambda \vdash[n]} (q-1)^{n-\ell(\Lambda)} N_{\Lambda,e}(q)$ using our calculations coincide with those given by Evseev in \cite{E}, which gives at least some indication that our methods yield correct results.

\item[(ii)] The set partition in the second part of the theorem is not one of the problem cases mentioned above; the formulas given for $N_{\Lambda,e}(q)$ come directly from Evseev's unmodified algorithm.
\end{enumerate}
\end{remarks}

When $n=14$ the number of problem cases increases by an order of magnitude, so verifying Conjecture \ref{ours} computationally  requires more powerful algorithms.
   Theorem \ref{products} and Table \ref{tbl1} suggest that the computation of the functions $N_{\Lambda,e}(q)$ for all $\Lambda\vdash[n]$ should be tractable, however, for at least the first few values of $n>13$.   

\subsection{Polynomial formulas for $N_{n,e}(q)$ with $e\leq 8$}

Recall that $N_{n,e}(q)$ denotes the number of irreducible characters of $\UT_n(q)$ with degree $q^e$. Here we show how the computations described in the last section allow us to derive  bivariate polynomials in $n,q$ giving $N_{n,e}(q)$ for small values of $e$.


\def\wtN{\wt N}
\def\tbf{\textbf}

In this direction, we first describe a useful intermediate formula for $N_{n,e}(q)$.  For $n\geq 2$ and $e \in \ZZ$, define
\be\label{tilde-def}\ba \wt N_{1,e}(q) =\  & \begin{cases}q,&\text{if }e=0, \\ 0,&\text{otherwise},\end{cases}
\\
\\[-10pt]
 \wtN_{n,e}(q) =\ &\text{the number of irreducible characters of $\UT_{n+1}(q)$ of degree $q^e$} \\
&\text{appearing as constituents of supercharacters whose shapes have} \\
&\text{a crossing-connected component which involves both 1 and $n+1$.}\ea\ee
%
In other words, $\wtN_{n,e}(q)$ for $n\geq 2$  is the sum \[\wtN_{n,e}(q) = \sum_\Lambda (q-1)^{n+1-\ell(\Lambda)} N_{\Lambda,e}(q)\] over all set partitions $\Lambda \vdash[n+1]$ which have a sequence of arcs $(i_t,j_t) \in \Arc(\Lambda)$, $t=1,\dots,k$, such that $i_1=1$ and $j_k=n+1$ and  $i_t<i_{t+1} < j_t < j_{t+1}$ for all $t$.

In the following statement,  we recall that a composition $\tbf c$ of an integer $x$ is a sequence of positive integers with $\sum_i \tbf c_i = x$, while a weak composition $\tbf w$ of $x$ is a sequence of nonnegative integers with $\sum_i \tbf w_i = x$.  We denote the number of elements in the sequences giving $\tbf c$ and $\tbf w$ by $\ell(\tbf c)$ and $\ell(\tbf w)$, respectively.

\begin{theorem} \label{tilde} Fix a prime power $q>1$.  Then for all integers $n\geq 2$ and $e\geq 0$, the number $N_{n,e}(q)$ of irreducible characters of $\UT_n(q)$ of degree $q^e$ is equal to
\[ N_{n,e}(q) = \sum_{(\tbf c, \tbf w)} \prod_{i=1}^{\ell( \tbf c )}\wtN_{\textbf{c}_i,\tbf w_i}(q)
,
\] where the sum is over all pairs $(\tbf c, \tbf w)$ with $\tbf c$  a composition of $n-1$ and $\tbf w$ a weak composition of $e$ such that $\ell(\tbf c) = \ell(\tbf w)$.  
\end{theorem}

\begin{proof}
In the standard representation of a set partition $\Lambda \vdash[n]$,
draw vertical lines through each vertex, and let the sequence of integers $1 = a_0 < a_1 <\dots < a_\ell = n$ index the vertices at which these lines do not intersect any arcs of $\Lambda$.  For example, if $\Lambda\vdash[13]$ is given by
\be\label{ex-outline} \barr{c} \\ \\  \\[-40pt]
\xy<0.25cm,0.8cm> \xymatrix@R=-0.0cm@C=.5cm{
*{} & *{} & *{} & *{} & *{} & *{} & *{} & *{} & *{} & *{} & *{} & *{} & *{}
\\
\\
\\ 
\\
\\
*{\bullet} \ar @{-} [uuuuu] \ar @/^1.2pc/ @{-} [rr] &
*{\bullet}  \ar @{-} [uuuuu]\ar @/^1.6pc/ @{-} [rrr] &
*{\bullet}  \ar @{-} [uuuuu]&
*{\bullet} \ar @{-} [uuuuu]&
*{\bullet} \ar @{-} [uuuuu] &
*{\bullet} \ar @{-} [uuuuu] \ar @/^.8pc/ @{-} [r] &
*{\bullet}\ar @{-} [uuuuu]&
*{\bullet}\ar @{-} [uuuuu]&
*{\bullet}\ar @{-} [uuuuu] &
*{\bullet}\ar @{-} [uuuuu]\ar @/^1.6pc/ @{-} [rrr] &
*{\bullet}\ar @{-} [uuuuu]\ar @/^.8pc/ @{-} [r] &
*{\bullet}\ar @{-} [uuuuu]&
*{\bullet} \ar @{-} [uuuuu]&
\\
1   & 
2 &
3  &
4 &
5 &
6 &
7 &
8 &
9 &
10 &
11 &
12 &
13 
}\endxy\\
\Lambda= \{ \{ 1,3\}, \{2,5\}, \{4\}, \{6,7\}, \{8\}, \{9\}, \{10,13\}, \{11,12\} \} \vdash[13]\\ \earr\ee then $(a_0,a_1,\dots,a_7) = (1,5,6,7,8,9,10,13)$.
Call the sequence $a=(a_0,a_1,\dots,a_\ell)$ the \emph{outline} of $\Lambda$.  
 Given a set partition $\Lambda \vdash[n]$ with outline $a= (a_0,a_1,\dots, a_\ell)$,  define $\Lambda(i)$ for $i=1,\dots,\ell$ as the set partition of $[a_i-a_{i-1}+1]$ formed by intersecting each part of $\Lambda$ with the interval $[a_{i-1},a_i]$, excluding all instances of the empty set, and then standardizing the result.  
  E.g., in (\ref{ex-outline}) we have 
 \[ \ba  \Lambda(1)&= \{ \{1,3\},\{2,5\}, \{4\}\},
 \\
  \Lambda(2)  =\Lambda(4) =\Lambda(5) =\Lambda(6)&= \{ \{1\}, \{2\}\},
  \ea\qquad\ba
  \Lambda(3) & = \{ \{1,2\} \},
  \\
        \Lambda(7)  &= \{ \{1,4\}, \{2,3\} \}.\ea
        \] 
      
      The vertical line through vertex $i$ intersects no arcs of $\Lambda$ if and only if there is no  $(x,y) \in \Arc(\Lambda)$ with $x<i<y$.  This observation has two consequences.  First, it shows that  $\Lambda(i)$ 
   must have a crossing-connected component involving    both 1 and $a_i-a_{i-1}+1$ whenever  $a_i-a_{i-1}>1$.  Second, 
each crossing-connected component of $\Lambda$  must partition a subset of one of the intervals $[a_{i-1},a_i]$, and so  it follows from Theorem \ref{products}  that 
        $ N_{\Lambda,e}(q) = \sum_{\textbf{w}} \prod_{i=1}^\ell N_{\Lambda(i),\textbf{w}_i}(q),$
         where the sum is over all weak compositions $\textbf{w}$ of $e$ with $\ell$ parts.  Also, since $n-\ell(\Lambda)$ is the cardinality of $\Arc(\Lambda)$, we have $n-\ell(\Lambda) = \sum_{i=1}^\ell \(a_i-a_{i-1}+1-\ell(\Lambda(i))\)$.
                
        Let $\mathscr{S}_k$ for $k\geq 2$ be the set of all set partitions of $[k+1]$ with a crossing-connected component involving both 1 and $k+1$, and let $\mathscr{S}_1$ be the set whose two elements are the distinct set partitions of $\{1,2\}$.  Fix a sequence $1=a_0 < a_1<\dots < a_\ell = n$ and write $n_i = a_i-a_{i-1}$.    It is apparent that the map 
\[\barr{ccc} 
 \Bigl\{\text{ Set partitions of $[n]$ with outline $(a_0,a_1,\dots,a_\ell)$ }\Bigr\} & \to & \mathscr{S}_{n_1} \times \mathscr{S}_{n_2} \times \cdots \times \mathscr{S}_{n_\ell} \\
 \Lambda & \mapsto & \(\Lambda(1),\ \Lambda(2),\ \dots,\ \Lambda(\ell)\)\earr\] is a bijection.  
  Combining this with the observations in the previous paragraph, one deduces that the sum of $(q-1)^{n-\ell(\Lambda)} N_{\Lambda,e}(q)$ over all $\Lambda \vdash[n]$ with outline $(a_0,a_1,\dots, a_\ell)$ is
 \[ \sum_{\Lambda} (q-1)^{n-\ell(\Lambda)} N_{\Lambda,e}(q) = \sum_{\textbf{w}} \prod_{i=1}^\ell\( \sum_{\Gamma \in \mathscr{S}_{n_i}} (q-1)^{n_i+1 -\ell(\Gamma)} N_{\Gamma, \textbf{w}_i}(q)\),\] where the outer sum is over all weak compositions $\textbf{w}$ of $e$ with $\ell$ parts.     If $\Gamma \vdash[2]$ then $N_{\Gamma,e}(q) = \delta_{e0}$, and noting this, one sees that the parenthesized sum is precisely $\wtN_{n_i,\textbf{w}_i}(q)$.
By summing the preceding equation over all possible outlines $(a_0,a_1,\dots,a_\ell)$, we obtain
$ N_{n,e}(q) = \sum_{(a,\textbf{w})} \prod_{i=1}^{\ell(\textbf{w})} \wtN_{a_{i}-a_{i-1}, \textbf{w}}(q)
$
 where the sum is over all pairs $(a,\textbf{w})$ where $a = (a_0,a_1,\dots,a_\ell)$ is a sequence of integers with $1=a_0<a_1<\dots<a_\ell = n$ and $\textbf{w}$ is a weak composition of $e$ with $\ell=\ell(\textbf{w})$ parts.  The theorem now follows by noting that the map $a\mapsto (a_1-a_0, a_2-a_1,\dots, a_{\ell}-a_{\ell-1})$ defines a bijection from possible outlines of $\Lambda \vdash[n]$ to compositions of $n-1$ with $\ell$ parts.
\end{proof}

When $e\leq 8$, we can show that $\wtN_{n,e}(q) =0$ for all but finitely many values of $n$; we suspect but cannot prove that the same is true for all values of $e$.  
Since Evseev's algorithm with Theorem \ref{ref}    allows us to compute the nonzero functions $\wtN_{n,e}(q)$,   the 
preceding result will  thus determine a  formula in $n$ and $q$ for $N_{n,e}(q)$. 
In this direction, we first make the following elementary observation.

\begin{observation}
If $\Lambda$ is any set partition then $N_{\Lambda,e}(q) = 0$ whenever $ d(\Lambda) - |\Cr(\Lambda)| > e$.  
\end{observation}

\begin{proof}
By (\ref{fact}) and the remarks in Section \ref{supercharacters}, if $\chi$ is a supercharacter of $\UT_n(q)$ with shape $\Lambda$ then each irreducible constituent of $q^{d(\Lambda)-|\Cr(\Lambda)|} \chi$ appears with multiplicity equal to its degree; however, this multiplicity is obviously at least $q^{d(\Lambda)-|\Cr(\Lambda)|}$.
\end{proof}

For any given integer $e\geq 0$, there are still an infinite number of set partitions $\Lambda$ with $d(\Lambda) - |\Cr(\Lambda)| \leq e$ and $N_{\Lambda,e}(q)\neq 0$.  Nevertheless, for $e\leq 8$, there are only a finite number of crossing-connected set partitions with $N_{\Lambda,e}(q)\neq 0$.  To prove this  we depend on the following technical lemma.

\begin{lemma}
Fix an integer $f\geq 0$ and suppose that for all integers $e,n$ with $0\leq e< \frac{n-3}{2}\leq f$, we have
$N_{\Lambda,e}(q) =0$ whenever $\Lambda\vdash[n]$ is crossing-connected.  Then for all integers $e,n$ with $e\leq f$ and $n>2e+2$, we have $N_{\Lambda,e}(q) =0$ whenever $\Lambda \vdash[n]$ is crossing-connected.
\end{lemma}

\begin{proof}
For any set partition $\Lambda$ let $m(\Lambda)$ be the least integer such that $N_{\Lambda,m(\Lambda)}(q) \neq 0$. By assumption, $m(\Lambda) \geq \lfloor \frac{n-1}{2} \rfloor$ if $\Lambda \vdash[n]$ is crossing-connected and $n\leq 2f+3$. 
Let $n > 2f+3$ and choose a crossing-connected set partition $\Lambda\vdash[n]$.  Suppose $m(\Gamma) \geq f+1$ if $\Gamma\vdash[n']$ is crossing-connected and $2f+3\leq n' < n$; to prove the lemma it suffices by induction to show that $m(\Lambda) \geq f+1$.

Because $\Lambda$ is crossing-connected, every $j \in [n]$ must be involved in some arc of $\Lambda$, yet no arcs have the form $(j,j+1)$, since such arcs cannot be involved in crossings.  This implies that $(a,n) \in \Arc(\Lambda)$ for some $a \in [n-2]$. 
Let $\Gamma \vdash[n-1]$ be the set partition formed by deleting the vertex $n$ and the arc $(a,n)$ from the standard representation of $\Lambda$, and define $\lambda,\kappa,\gamma \in \fkt_{n}(q)^*$ as the maps given by 
\[\lambda = \sum_{(i,j) \in \Arc(\Lambda)} e_{i,j}^*\qquad\text{and}\qquad
 \kappa = e_{a,n}^*\qquad\text{and}\qquad \gamma = \lambda -\kappa.\] 
Then $N_{\Lambda,e}(q)$ is the number of constituents of the supercharacter $\chi_\lambda$ of degree $q^e$; $N_{\Gamma,e}(q)$ is the number of constituents of the supercharacter $\chi_\gamma$ of degree $q^e$ by Lemma \ref{lem2};
  and $\chi_\lambda = \chi_\gamma\otimes \chi_\kappa$ by Lemma \ref{form}.
Let $\psi$ be an irreducible constituent of $\chi_\gamma$ of degree $q^e$.  Then $e \geq m(\Gamma)$ and $\psi$ appears in $\chi_\gamma$ with multiplicity $q^{e- (d(\Gamma)-|\Cr(\Gamma)|)}$ since $q^{d(\Gamma)-|\Cr(\Gamma)|}$ is the multiplicity of $\chi_\gamma$ in the regular representation of $\UT_n(q)$.  The possibly reducible product $\psi \otimes \chi_\kappa$ therefore appears in
$\chi_\lambda$ with multiplicity $q^{e- (d(\Gamma)-|\Cr(\Gamma)|)}$ and in the regular representation of $\UT_n(q)$ with multiplicity  $q^{e+e'}$ where $e'=(d(\Lambda) - |\Cr(\Lambda)|)- (d(\Gamma)-|\Cr(\Gamma)|)$. 
It is apparent from the definitions (\ref{cr-d}) that $e'\geq 0$, and so every irreducible constituent of $\psi \otimes \chi_\kappa$ has degree at least $q^e$.  Since every irreducible constituent of $\chi_\lambda$ is a constituent of some such product $\psi \otimes \chi_\kappa$, it follows that $m(\Lambda) \geq m(\Gamma)$.

Therefore, it suffices to show that $m(\Gamma) \geq f+1$.  This is not immediate by hypothesis because $\Gamma$ is not necessarily crossing-connected.
Let $\Gamma_1,\dots,\Gamma_k$ be the crossing-connected components of $\Gamma$; let $\Gamma_i'$ be the standardization of $\Gamma_i$; and let $\cS_i$ be the subset of $[n-1]$ and $\ell_i$ the positive integer such that $\Gamma_i \vdash[\cS_i]$ and $\Gamma_i' \vdash[\ell_i]$.  
  After we recall the notation in (\ref{arb-def}), it follows by Theorem \ref{products}, that 
\[m(\Gamma) = \sum_{i=1}^k m(\Gamma_i) = \sum_{i=1}^k m(\Gamma_i') + \sum_{i=1}^k \(d(\Gamma_i) - d(\Gamma_i')\).\]  Recall that $(a,n) \in \Arc(\Lambda)$ for some $a \in [n-2]$.
Necessarily $a\in \cS_i$ for exactly one $i \in [k]$; we may assume $i=1$.  Then the arc $(a,n)$  must cross at least one arc in every remaining crossing-connected component $\Gamma_2,\dots, \Gamma_k$ of $\Gamma$.
Each $\Gamma_i$ for $i>1$ therefore has an arc of the form $(x,y)$ with $x<a<y$, yet $a \notin \cS_i$, so by definition $d(\Gamma_i) - d(\Gamma_i') \geq 1$.  Hence $\sum_{i=1}^k \(d(\Gamma_i) - d(\Gamma_i')\) \geq k-1$.  

By hypothesis 
\[ \sum_{i=1}^k m(\Gamma_i') \geq \sum_{i=1}^k \min \left\{ \left \lfloor \frac{\ell_i-1}{2} \right\rfloor , f+1\right\} \geq \min \left\{ \sum_{i=1}^k \left\lfloor \frac{\ell_i-1}{2} \right\rfloor, f+1 \right\}.\]  
 Let $t$ be the number of $\ell_i$'s which are odd.  Then 
 $ \sum_{i=1}^k \left\lfloor \frac{\ell_i-1}{2} \right\rfloor = \frac{1}{2}\(t+\sum_{i=1}^k \ell_i\)  - k$ and it follows that 
 \be\label{f+1} m(\Gamma) \geq \min \left\{ \frac{1}{2}\(t+\sum_{i=1}^k \ell_i\) - 1,f+1 \right\}.\ee
We must have $\sum_{i=1}^k \ell_i \geq n-1 \geq 2f+3$ since $\bigcup_{i=1}^k \cS_i = [n-1]$.  As $t$ determines the parity of $\sum_{i=1}^k \ell_i$, one checks that the right hand side of (\ref{f+1}) is $\geq f+1$, as required.
\end{proof}

We apply this lemma to the result of the following explicit computation.  It is a time-consuming but tractable problem for a computer to enumerate the crossing-connected set partitions $\Lambda \vdash[n]$ satisfying $d(\Lambda) - |\Cr(\Lambda)| \leq 8$ for $n\leq 19$.  Evseev's algorithm \cite{implement,implement2} fortunately succeeds in computing polynomial formulas $N_{\Lambda,e}(q)$ for all such set partitions $\Lambda$, and by inspecting these formulas we are able to deduce that 
$N_{\Lambda,e}(q) = 0$ whenever $\Lambda \vdash[n]$ is crossing-connected and $e \leq \frac{n-3}{2} \leq 8$.  Taking $f=8$ in the preceding lemma gives the following:

\begin{proposition}\label{instrumental}
Let $e\leq 8$ be a nonnegative integer. 
\begin{enumerate}
\item[(1)] If $n>2e+2$, then $N_{\Lambda,e}(q)=0$ for all crossing-connected set partitions $\Lambda \vdash[n]$.

\item[(2)] If $n>2e+1$ then $\wtN_{n,e}(q) = 0$.
\end{enumerate}
\end{proposition}

\begin{remark} 
One can presumably extend this result by repeating our calculations with a larger integer in place of eight.  It seems reasonable, in fact, to conjecture that the proposition holds for all nonnegative integers $e$.
\end{remark}

\begin{proof}
Part (1) is immediate.  Let $n>2e+1$ and suppose $\Lambda\vdash[n+1]$ has a crossing-connected component $\Gamma$ involving both 1 and $n+1$. To prove (2), it suffices by Theorem \ref{products} to show that  $N_{\Gamma,f}(q)=0$ for all $f \leq e$.  To this end,  suppose the standardization $\st(\Gamma)$ partitions the set $[k]\subseteq [n+1]$.  Then  $d(\Gamma) - d(\st(\Gamma)) \geq n+1-k$ and it follows from (1) and the definition (\ref{arb-def}) that $N_{\Gamma,f}(q) =0$ if $k > 2\(f-(n+1-k)\) + 2$ or equivalently if $n + (n+1-k) > 2f+1$.  This inequality holds for all $f\leq e$ since $ n+1-k\geq 0$ and $n>2e+1$.
\end{proof}

Similarly, it is a feasible computer calculation to enumerate all set partitions $\Lambda \vdash[n+1]$  for $n\leq 17$ which satisfy  $d(\Lambda) - |\Cr(\Lambda)| \leq 8$ and have a crossing-connected component involving both 1 and $n+1$.  Evseev's algorithm \cite{implement,implement2} succeeds in computing polynomial formulas $N_{\Lambda,e}(q)$ for all such $\Lambda$, and establishes in addition that all irreducible characters of the crossing algebra groups $1+\wt\fkcr_\Lambda(q)$ are Kirillov functions.  
This computation determines the nonzero polynomials $\wtN_{n,e}(q)$ with $e\leq 8$; we list these in Appendix \ref{app-a}.  When written as functions of $q-1$, these polynomials turn out to have nonnegative integer coefficients, and so by Theorem \ref{tilde} we may conclude that the same is true of $N_{n,e}(q)$ for all integers $n\geq 1$ and $e\leq 8$.

We use this data to prove Theorem \ref{intro} in the following way.
  Suppose $\tbf d$ is a composition of a positive integer $e$ with $\ell$ parts.  There are exactly $\binom{k + \ell}{\ell}$ ways of adding $k$ zeros to $\tbf d$ to form a weak composition of $e$ with $k+\ell$ parts, since these extensions are in bijection with the weak compositions of $k$ with $\ell+1$ parts.  Given a composition $\tbf c$, write $|\tbf c|$ to  denote the sum of its parts.  
%
%
Since $\wtN_{n,0}(q)=0$ if $n>1$ and $\wtN_{1,0}(q)=q$,  it follows that for integers $e,n\geq 1$, we may rewrite the formula for $N_{n,e}(q)$  in Theorem \ref{tilde}  as
\be\label{sum}
N_{n+1,e}(q) 
=
\sum_{(\tbf c, \tbf d)} \binom{n-|\tbf c| + \ell(\tbf c)}{ \ell(\tbf c)} q^{n-|\tbf c|} \prod_{i=1}^{\ell( \tbf c )}\wtN_{\tbf c_i,\tbf d_i}(q)\ee
where the sum is over all pairs of compositions $(\tbf c, \tbf d)$ such that $|\tbf c| \leq n$ and $|\tbf d| = e$ and $\ell(\tbf c) = \ell(\tbf d)$.
Suppose $e \in \{1,\dots,8\}$.  Since $\wtN_{n,e}(q)=0$ when $n$ is sufficiently large, there are only finitely many pairs $(\tbf c, \tbf d)$ indexing nonzero terms in the sum (\ref{sum}).  Since we have polynomial formulas for the functions $\wtN_{\tbf c_i, \tbf d_i}(q)$, we can thus determine bivariate polynomials in $n,q$ giving each nonzero summand in (\ref{sum}).  Summing these functions then gives a formula for $N_{n,e}(q)$ that is valid when $n$ is large enough, and which happens to have the form described in  Theorem \ref{intro}.

This discussion affords a proof of the following theorem, whose statement combines Theorems \ref{intro} and \ref{intro2} from the introduction.

%
%

\begin{theorem} \label{summary} Fix a prime power $q>1$, a positive integer $n$, and an integer $e \in \{1,\dots,8\}$.

\begin{enumerate}

\item[(1)] $N_{n,e}(q)$ is a polynomial in $q-1$ with nonnegative integer coefficients.

\item[(2)] There are polynomials $f_{e,i}(x)$ with nonnegative integer coefficients such that if $n>2e$ then
\[ N_{n,e}(q) = q^{n-e-2}  \sum_{i=1}^{2e}\frac{c_{e,i}!}{e!} \cdot f_{e,i}(n-2e-1)\cdot  (q-1)^i,\quad\text{where $c_{e,i} = \tfrac{1}{2} + \left| \tfrac{1}{2} +e -i\right|$.}\]
 
 \item[(3)] Every irreducible character of $\UT_n(q)$ with degree $\leq q^8$ is a Kirillov function.
 
 \end{enumerate}
 \end{theorem}
 
 We list the polynomials $f_{e,i}(x)$ in Appendix \ref{app-b}.
The only thing not yet proved here is part (3), and this will follow from a short lemma.  
Define $\mathscr{S}_n$ for $n\geq 2$ as in the proof of Theorem \ref{tilde}:
\[\ba \mathscr{S}_1=\ & \text{the set of set partitions of $\{1,2\}$,}\\
\\[-10pt]
\mathscr{S}_n=\ &\text{the set of set partitions of $[n+1]$ which have a crossing-connected} \\
&\text{component which involves both 1 and $n+1$.}\ea
\]  Of course, $\wtN_{n,e}(q)$ is by definition the number of irreducible characters with degree $q^e$ which appear as constituents of supercharacters with shapes in $\mathscr{S}_n$.  Proposition \ref{instrumental} shows that for $e \leq 8$ there are only a finite number of set partitions $\Lambda \in \bigcup_n \mathscr{S}_n$ with $N_{\Lambda,e}(q)\neq 0$, and as remarked above, Evseev's algorithm establishes that for all such $\Lambda$, the irreducible characters of the algebra group $1+\wt\fkcr_\Lambda(q)$ are
 Kirillov functions.  Thus, the following result proves (3) in our theorem.

\begin{lemma}
Fix a prime power $q>1$ and an integer $e\geq 0$.  
Suppose whenever $f\leq e$ and $\Lambda \in \bigcup_n \mathscr{S}_n$ has $N_{\Lambda,f}(q) \neq 0$, all irreducible characters of the algebra group $1+\wt\fkcr_\Lambda(q)$ are Kirillov functions.
  Then the irreducible characters of $\UT_n(q)$ with degree $\leq q^e$ are Kirillov functions for all positive integers $n$.
\end{lemma}

\begin{proof}
Assume our hypothesis and recall the notation used in the proof of Theorem \ref{tilde}.  Fix a set partition $\Lambda \vdash[n]$ with outline $(a_0,a_1,\dots,a_\ell)$. Suppose  $\lambda \in \fkt_n(q)^*$ is quasi-monomial with shape $\Lambda$.  To prove the lemma, it suffices to show that all irreducible constituents with degree $\leq q^e$ of $\chi_\lambda$ are Kirillov functions.

For each $i=1,\dots,\ell$, define $\Gamma_i$ as the set partition of $[n]$ formed by adding singleton parts to the set partition $(a_{i-1}-1) + \Lambda(i)$ of $[a_{i-1},a_i]$.  Likewise, let  $\gamma_i \in \fkt_n(q)^*$ for $i=1,\dots,\ell$ be the quasi-monomial map with shape $\Gamma_i$ given by 
\[ \gamma_i = \sum_{(j,k) \in\Arc(\Gamma_i)} \lambda_{jk} e_{jk}^* \in \fkt_n(q)^*.\] 
Then $\lambda = \sum_{i=1}^\ell \gamma_i$, and it follows from Lemmas \ref{obs-las} and \ref{form} that every irreducible constituent $\psi$ of the supercharacter $\chi_\lambda$ has a unique factorization as $\psi = \psi_1\otimes \psi_2\otimes \cdots \otimes \psi_\ell$ where $\psi_i$ is an irreducible constituent of $\chi_{\gamma_i}$.  
Suppose $\psi(1) \leq q^e$, so that each $\psi_i(1) \leq q^e$.  The crossing algebra of each $\Gamma_i$  is certainly isomorphic to that of $\Lambda(i)$, and so  by assumption the irreducible characters  $1+\wt\fkcr_{\Gamma_i}(q)$ are Kirillov functions.  Therefore by  Theorem \ref{ref} each $\psi_i$ is a Kirillov function, and it follows  by Lemma \ref{obs-las} that $\psi$ is a Kirillov function, as required.\end{proof}

\appendix

\section{The polynomials $\wt N_{n,e}(q)$ in Theorem \ref{tilde}}\label{app-a}

The nonzero polynomials $\wtN_{n,e}(q)$ for $e\leq 8$ are listed below in 
Tables \ref{tbl2} and \ref{tbl3}.  
These polynomials have at least one curious property worth taking the trouble to point out.
Let $A(n,k)$ and $B(n,k)$ define the following triangular arrays, given as sequences {A026374}  and  {A026386} in \cite{OEIS}:
\[ \ba A(n,k) &= \begin{cases} 1,&\text{if }k=0\text{ or }k=n, \\
A(n-1,k-1)+A(n-1,k),&\text{if $n$ is odd and $1\leq k \leq n-1$,} \\
A(n-1,k-1)+A(n-1,k)+A(n-2,k-1),&\text{if $n$ is even and $1\leq k \leq n-1$,}\\
0,&\text{otherwise},\end{cases}
\\ \\[-10pt]
B(n,k) &= \begin{cases} 1,&\text{if }k=0\text{ or }k=n, \\
B(n-1,k-1)+B(n-1,k)+B(n-2,k-1),&\text{if $n$ is odd and $1\leq k \leq n-1$,} \\
B(n-1,k-1)+B(n-1,k),&\text{if $n$ is even and $1\leq k \leq n-1$,}\\
0,&\text{otherwise}.\end{cases}
\ea\]
More tangibly, $A(n,k)$ is the number of integer sequences $(s_0,s_1,\dots,s_n)$ with $s_i-s_{i-1} \in \{-1,0,1\}$ such that $s_0=0$, $s_n = n-2k$, and $s_i$ is even if $i$ is even.  Likewise,  $B(n,k)$ is the number of integer sequences $(s_0,s_1,\dots,s_n)$ with $s_i-s_{i-1} \in \{-1,0,1\}$ such that $s_0=0$, $s_n = n-2k$, and $s_i$ is odd if $i$ is odd.
One checks from our tables that the following happens to hold.
   
\begin{observation} Let $e\in \{1,\dots,8\}$ and let $q>1$ be a prime power.  \begin{enumerate}
\item[(1)] If $n=2e+1$ then $\wtN_{n,e}(q) = \sum_{k=0}^{n-2} A(n-2,k) (q-1)^{n-e+k}$.
\item[(2)] If $n=2e$ then $\wtN_{n,e}(q) = \sum_{k=0}^{n-2} \(A(n-2,k)+(e-1) B(n-2,k)\) (q-1)^{n-e+k}$.
\end{enumerate}
\end{observation}

\begin{table}[H]
\[ \barr{|c|c|l|} \hline n & e & [\ a_0,\ a_1,\ \dots,\ a_\ell\ ]\text{ where }\wtN_{n,e}(q) = \sum_{i=0}^\ell a_i (q-1)^i \\
\hline 
1&0 & [\ 1,\ 1\ ]\\
\hline 
2&1 & [\ 0,\ 1\ ]\\
\hline 
3&1 & [\ 0,\ 0,\ 1,\ 1\ ]\\
3&2 & [\ 0,\ 1,\ 1\ ]\\
\hline 
4&2 & [\ 0,\ 0,\ 2,\ 5,\ 2\ ]\\
4&3 & [\ 0,\ 1,\ 3,\ 2\ ]\\
4&4 & [\ 0,\ 0,\ 1\ ]\\
\hline 
5&2 & [\ 0,\ 0,\ 0,\ 1,\ 4,\ 4,\ 1\ ]\\
5&3 & [\ 0,\ 0,\ 3,\ 12,\ 13,\ 4\ ]\\
5&4 & [\ 0,\ 1,\ 5,\ 13,\ 10,\ 2\ ]\\
5&5 & [\ 0,\ 0,\ 3,\ 5,\ 2\ ]\\
5&6 & [\ 0,\ 0,\ 1,\ 1\ ]\\
\hline 
6&3 & [\ 0,\ 0,\ 0,\ 3,\ 16,\ 27,\ 16,\ 3\ ]\\
6&4 & [\ 0,\ 0,\ 4,\ 23,\ 50,\ 46,\ 17,\ 2\ ]\\
6&5 & [\ 0,\ 1,\ 7,\ 31,\ 59,\ 44,\ 13,\ 1\ ]\\
6&6 & [\ 0,\ 0,\ 5,\ 26,\ 40,\ 23,\ 4\ ]\\
6&7 & [\ 0,\ 0,\ 3,\ 12,\ 15,\ 5\ ]\\
6&8 & [\ 0,\ 0,\ 1,\ 4,\ 3\ ]\\
\hline 
7&3 & [\ 0,\ 0,\ 0,\ 0,\ 1,\ 7,\ 17,\ 17,\ 7,\ 1\ ]\\
7&4 & [\ 0,\ 0,\ 0,\ 6,\ 41,\ 100,\ 106,\ 49,\ 8\ ]\\
7&5 & [\ 0,\ 0,\ 5,\ 38,\ 129,\ 232,\ 211,\ 97,\ 22,\ 2\ ]\\
7&6 & [\ 0,\ 1,\ 9,\ 57,\ 187,\ 288,\ 221,\ 85,\ 15,\ 1\ ]\\
7&7 & [\ 0,\ 0,\ 7,\ 58,\ 182,\ 257,\ 174,\ 54,\ 6\ ]\\
7&8 & [\ 0,\ 0,\ 5,\ 46,\ 133,\ 162,\ 88,\ 19,\ 1\ ]\\
\hline
8&4 & [\ 0,\ 0,\ 0,\ 0,\ 4,\ 33,\ 102,\ 147,\ 102,\ 33,\ 4\ ]\\
8&5 & [\ 0,\ 0,\ 0,\ 10,\ 85,\ 289,\ 503,\ 473,\ 242,\ 68,\ 11,\ 1\ ]\\
8&6 & [\ 0,\ 0,\ 6,\ 57,\ 267,\ 720,\ 1083,\ 894,\ 401,\ 91,\ 8\ ]\\
8&7 & [\ 0,\ 1,\ 11,\ 91,\ 429,\ 1102,\ 1575,\ 1296,\ 626,\ 177,\ 28,\ 2\ ]\\
8&8 & [\ 0,\ 0,\ 9,\ 102,\ 500,\ 1218,\ 1601,\ 1178,\ 483,\ 103,\ 9\ ]\\
\hline 
9&4 & [\ 0,\ 0,\ 0,\ 0,\ 0,\ 1,\ 10,\ 39,\ 75,\ 75,\ 39,\ 10,\ 1\ ]\\
9&5 & [\ 0,\ 0,\ 0,\ 0,\ 10,\ 97,\ 372,\ 720,\ 750,\ 420,\ 118,\ 13\ ]\\
9&6 & [\ 0,\ 0,\ 0,\ 15,\ 154,\ 687,\ 1724,\ 2592,\ 2342,\ 1270,\ 409,\ 74,\ 6\ ]\\
9&7 & [\ 0,\ 0,\ 7,\ 80,\ 482,\ 1775,\ 3957,\ 5320,\ 4385,\ 2262,\ 749,\ 166,\ 25,\ 2\ ]\\
9&8 & [\ 0,\ 1,\ 13,\ 133,\ 822,\ 2968,\ 6374,\ 8317,\ 6747,\ 3449,\ 1110,\ 219,\ 24,\ 1\ ]\\ \hline 
\earr\]
\caption{Polynomials in $q-1$ giving $\wtN_{n,e}(q)$ for $0\leq e \leq 8$ and $1\leq n \leq 9$}
\label{tbl2}
\end{table}

\normalsize

\small

\def\skipline{\\ &&\ \ }

\begin{table}[H]
\[ \barr{|c|c|l|} \hline n & e & [\ a_0,\ a_1,\ \dots,\ a_\ell\ ]\text{ where }\wtN_{n,e}(q) = \sum_{i=0}^\ell a_i (q-1)^i 
\\ \hline 
10&5 & [\ 0,\ 0,\ 0,\ 0,\ 0,\ 5,\ 56,\ 254,\ 600,\ 795,\ 600,\ 254,\ 56,\ 5\ ]\\
10&6 & [\ 0,\ 0,\ 0,\ 0,\ 20,\ 226,\ 1066,\ 2735,\ 4171,\ 3895,\ 2245,\ 803,\ 181,\ 26,\ 2\ ]\\
10&7 & [\ 0,\ 0,\ 0,\ 21,\ 254,\ 1418,\ 4708,\ 9860,\ 13084,\ 10947,\ 5719,\ 1809,\ 318,\ 24\ ]\\
10&8 & [\ 0,\ 0,\ 8,\ 107,\ 792,\ 3740,\ 11303,\ 22002,\ 27905,\ 23448,\ 13309,\ 5212,\ 1435,\ 276,\ 34,\ 2\ ]\\
\hline 
11&5 & [\ 0,\ 0,\ 0,\ 0,\ 0,\ 0,\ 1,\ 13,\ 70,\ 202,\ 339,\ 339,\ 202,\ 70,\ 13,\ 1\ ]\\
11&6 & [\ 0,\ 0,\ 0,\ 0,\ 0,\ 15,\ 189,\ 994,\ 2836,\ 4791,\ 4935,\ 3096,\ 1146,\ 229,\ 19\ ]\\
11&7 & [\ 0,\ 0,\ 0,\ 0,\ 35,\ 455,\ 2587,\ 8456,\ 17477,\ 23637,\ 21098,\ 12434,\ 4810,\ 1193,\ 176,\ 12\ ]\\
11&8 & [\ 0,\ 0,\ 0,\ 28,\ 391,\ 2638,\ 11013,\ 30197,\ 55114,\ 67310,\ 55436,\ 31170,\ 12234,\ 3506,\ \skipline 784,\ 139,\ 17,\ 1\ ]\\
\hline 
12&6 & [\ 0,\ 0,\ 0,\ 0,\ 0,\ 0,\ 6,\ 85,\ 510,\ 1690,\ 3390,\ 4263,\ 3390,\ 1690,\ 510,\ 85,\ 6\ ]\\
12&7 & [\ 0,\ 0,\ 0,\ 0,\ 0,\ 35,\ 495,\ 3014,\ 10370,\ 22269,\ 31172,\ 28949,\ 17934,\ 7421,\ 2057,\ \skipline 384,\ 47,\ 3\ ]\\
12&8 & [\ 0,\ 0,\ 0,\ 0,\ 56,\ 827,\ 5551,\ 22287,\ 58785,\ 105295,\ 129141,\ 108335,\ 61744,\ 23499,\ \skipline 5736,\ 817,\ 52\ ]\\
[-10pt] &&\\\hline &&\\[-10pt]
13&6 & [\ 0,\ 0,\ 0,\ 0,\ 0,\ 0,\ 0,\ 1,\ 16,\ 110,\ 425,\ 1015,\ 1558,\ 1558,\ 1015,\ 425,\ 110,\ 16,\ 1\ ]\\
13&7 & [\ 0,\ 0,\ 0,\ 0,\ 0,\ 0,\ 21,\ 326,\ 2185,\ 8290,\ 19645,\ 30338,\ 31023,\ 20990,\ 9235,\ 2530,\ \skipline 391,\ 26\ ]\\
13&8 & [\ 0,\ 0,\ 0,\ 0,\ 0,\ 70,\ 1105,\ 7730,\ 31635,\ 84111,\ 152436,\ 192368,\ 170352,\ 105998,\ \skipline 46182,\ 13917,\ 2805,\ 346,\ 20\ ]\\
\hline
14&7 & [\ 0,\ 0,\ 0,\ 0,\ 0,\ 0,\ 0,\ 7,\ 120,\ 897,\ 3840,\ 10410,\ 18696,\ 22685,\ 18696,\ 10410,\ \skipline 3840,\ 897,\ 120,\ 7\ ]\\
14&8 & [\ 0,\ 0,\ 0,\ 0,\ 0,\ 0,\ 56,\ 953,\ 7142,\ 31066,\ 87093,\ 165429,\ 218240,\ 202198,\ \skipline 132015,\ 60703,\ 19620,\ 4452,\ 705,\ 74,\ 4\ ]\\
\hline
15&7 & [\ 0,\ 0,\ 0,\ 0,\ 0,\ 0,\ 0,\ 0,\ 1,\ 19,\ 159,\ 771,\ 2400,\ 5028,\ 7247,\ 7247,\ 5028,\ 2400,\ \skipline 771,\ 159,\ 19,\ 1\ ]\\
15&8 & [\ 0,\ 0,\ 0,\ 0,\ 0,\ 0,\ 0,\ 28,\ 517,\ 4218,\ 20034,\ 61485,\ 128109,\ 185636,\ 188894,\ \skipline 134907,\ 66915,\ 22488,\ 4872,\ 613,\ 34\ ]\\
\hline 
16&8 & [\ 0,\ 0,\ 0,\ 0,\ 0,\ 0,\ 0,\ 0,\ 8,\ 161,\ 1442,\ 7581,\ 25998,\ 61194,\ 101458,\ 119941,\ \skipline 101458,\ 61194,\ 25998,\ 7581,\ 1442,\ 161,\ 8\ ]\\
\hline 
17&8 & [\ 0,\ 0,\ 0,\ 0,\ 0,\ 0,\ 0,\ 0,\ 0,\ 1,\ 22,\ 217,\ 1267,\ 4872,\ 12999,\ 24731,\ 34016,\ 34016,\ \skipline 24731,\ 12999,\ 4872,\ 1267,\ 217,\ 22,\ 1\ ]\\
\hline
\earr\]
\normalsize
\caption{Polynomials in $q-1$ giving $\wtN_{n,e}(q)$ for $0\leq e \leq 8$ and $10\leq n \leq 17$}
\label{tbl3}
\end{table}

\normalsize

\newpage
\section{The polynomials $f_{e,i}(x)$ in Theorem \ref{intro}}\label{app-b}

The polynomials $f_{e,i}(x)$ appearing in Theorems \ref{intro} and \ref{summary} are listed  below in Tables \ref{tbl4} and \ref{tbl5}.  
As described in those results, these polynomials determine $N_{n,e}(q)$ but only when $n>2e$.  One can compute $N_{n,e}(q)$ when $n\leq 2e \leq 16$, however, by invoking Theorem \ref{tilde} with the data in the previous section.  Polynomials in $q$ giving $N_{n,e}(q)$ already appear in \cite{I} for $n\leq 9$ and in \cite{E} for $n\leq 13$.  For completeness, we give the remaining computable cases here:
\[\ba  N_{14,7}(q) &=  6q^{18} + 70q^{17} + 180q^{16} + 227q^{15} - 843q^{14} - 1893q^{13} + 2734q^{12} + 2451q^{11} 
\\&\quad - 4015q^{10} - 45q^{9} + 1792q^8 - 722q^7 + 48q^6 + 10q^5, \\
\\[-10pt]
N_{14,8}(q)& = 3q^{19} + 28q^{18} + 132q^{17} + 387q^{16} - 122q^{15} - 1974q^{14} - 1490q^{13} + 5970q^{12} 
\\&\quad+ 691q^{11} - 6739q^{10} + 1942q^9 + 2596q^8 - 1705q^7 + 276q^6 + 6q^5 - q^4, \\
\\[-10pt]
N_{15,8}(q)& = 8q^{20} + 44q^{19} + 309q^{18} + 475q^{17} - 228q^{16} - 3705q^{15} - 1877q^{14} + 11423q^{13}
\\&\quad - 478q^{12} - 13050q^{11} + 5754q^{10} + 4290q^9 - 3756q^8 + 808q^7 - 12q^6 - 5q^5, \\
\\[-10pt]
N_{16,8}(q)&= 13q^{21} + 106q^{20} + 451q^{19} + 846q^{18} - 718q^{17} - 6378q^{16} - 2156q^{15} + 20656q^{14}
\\&\quad - 3521q^{13} - 23888q^{12} + 13914q^{11} + 6304q^{10} - 7517q^9 + 1984q^8 - 81q^7 - 15q^6.\ea\]
When written as polynomials in $q-1$, these have nonnegative integer coefficients (but take up significantly more space).

As mentioned in the introduction, for each $e \in \{1,\dots,8\}$ the polynomials $f_{e,i}(x)$ for $i=1,\dots,2e$ have degrees $1,2,\dots,e,e,\dots,2,1$ and leading coefficients    \[N(e,1),\ N(e,2),\ \dots,\ N(e,e),\ N(e,e),\ \dots, N(e,2),\ N(e,1)\] where $N(m,k) =\frac{1}{k} \binom{m-1}{k-1} \binom{m}{k-1}$ denotes the Narayana numbers. Another interesting feature of these polynomials is that $f_{e,1}(x) = x+e$ for $e=1,\dots,8$.  
The formula for $N_{n,e}(q)$ in Theorem \ref{intro} thus shows that if $n>2e$ and $e\leq8$ then the value of the derivative $\frac{d}{dq} N_{n,e}(q)$ at $q=1$ is $\max(n-e-1,0)$.  Isaacs \cite{I} made the same observation in the cases when $n\leq 9$ and $e\geq 0$ is arbitrary, and in fact,  if Lehrer's conjecture holds then this phenomenon always occurs as a result of the following:

\begin{proposition}
If $n>0$ and $e\geq 0$ are integers such that $N_{n,e}(q)$ is a polynomial in $q$, then differentiating $N_{n,e}(q)$ with respect to $q$ and setting $q=1$ gives 
\[ \frac{d}{dq} N_{n,e}(q) \Big\vert_{q=1} =\begin{cases} n-e-1,&\text{if }0\leq e<n, \\ 0,&\text{otherwise}.\end{cases}\]
\end{proposition}

As noted by an anonymous referee, this statement is a consequence of the following fact.  Here $\delta_{e,0}$ is the Kronecker delta, equal to 1 if $e=0$ and 0 otherwise.

\begin{proposition}   
$N_{n,e}(q) \equiv \delta_{e,0}+ \max(0,n-e-1)(q-1)\mod (q-1)^2$ for all $n>0$ and $e\geq 0$.
\end{proposition}

\begin{proof}
This is true for $e=0$ since $N_{n,e}(q) = q^{n-1}$.  If $e>0$, then (\ref{imply}) and Proposition \ref{eval}, show that $N_{n,e}(q) \equiv f(n,e)(q-1)\mod (q-1)^2$ where $f(n,e)$ is the number of  $\Lambda\vdash [n]$ with $n-1$ parts and $d(\Lambda) = e$.  To compute $f(n,e)$, observe that $\Lambda\vdash[n]$ has $n-1$ parts if and only if $\Lambda$ has exactly one arc $(i,j)$, in which case $d(\Lambda) = j-i-1$. 
\end{proof}

\begin{table}[H]
\[
 \begin{array}{|c|c|l|} \hline
 e & i &
   [\ a_0,\ a_1,\ \dots,\ a_\ell\ ]\text{ where }f_{e,i}(x) = \sum_{i=0}^\ell a_i x^i 
\\ \hline
1 & 1 & [\ 1,\ 1\ ]\\
1 & 2 & [\ 0,\ 1\ ]\\
\hline 
2 & 1 & [\ 2,\ 1\ ]\\
2 & 2 & [\ 10,\ 9,\ 1\ ]\\
2 & 3 & [\ 4,\ 7,\ 1\ ]\\
2 & 4 & [\ 0,\ 1\ ]\\
\hline 
3 & 1 & [\ 3,\ 1\ ]\\
3 & 2 & [\ 63,\ 33,\ 3\ ]\\
3 & 3 & [\ 204,\ 149,\ 24,\ 1\ ]\\
3 & 4 & [\ 108,\ 110,\ 21,\ 1\ ]\\
3 & 5 & [\ 9,\ 18,\ 3\ ]\\
3 & 6 & [\ 0,\ 1\ ]\\
\hline 
4 & 1 & [\ 4,\ 1\ ]\\
4 & 2 & [\ 220,\ 82,\ 6\ ]\\
4 & 3 & [\ 2448,\ 1194,\ 156,\ 6\ ]\\
4 & 4 & [\ 6720,\ 4010,\ 695,\ 46,\ 1\ ]\\
4 & 5 & [\ 3984,\ 2886,\ 575,\ 42,\ 1\ ]\\
4 & 6 & [\ 516,\ 510,\ 108,\ 6\ ]\\
4 & 7 & [\ 16,\ 34,\ 6\ ]\\
4 & 8 & [\ 0,\ 1\ ]\\
\hline 
5 & 1 & [\ 5,\ 1\ ]\\
5 & 2 & [\ 565,\ 165,\ 10\ ]\\
5 & 3 & [\ 14300,\ 5460,\ 600,\ 20\ ]\\
5 & 4 & [\ 113160,\ 52150,\ 7790,\ 470,\ 10\ ]\\
5 & 5 & [\ 283560,\ 149414,\ 26505,\ 2085,\ 75,\ 1\ ]\\
5 & 6 & [\ 183120,\ 107864,\ 21050,\ 1815,\ 70,\ 1\ ]\\
5 & 7 & [\ 31680,\ 21240,\ 4370,\ 360,\ 10\ ]\\
5 & 8 & [\ 1940,\ 1640,\ 340,\ 20\ ]\\
5 & 9 & [\ 40,\ 60,\ 10\ ]\\
5 & 10 & [\ 0,\ 1\ ]\\
\hline
\end{array}
\]
\caption{Polynomials $f_{e,i}(x)$ for $1\leq e\leq 5$ and $1\leq i\leq 2e$}
\label{tbl4}
\end{table}

\normalsize

\begin{table}[H]
\small
\[
 \barr{|c|c|l|} \hline e & i &
   [\ a_0,\ a_1,\ \dots,\ a_\ell\ ]\text{ where }f_{e,i}(x) = \sum_{i=0}^\ell a_i x^i 
\\\hline 
6 & 1 & [\ 6,\ 1\ ]\\
6 & 2 & [\ 1212,\ 291,\ 15\ ]\\
6 & 3 & [\ 59130,\ 18475,\ 1725,\ 50\ ]\\
6 & 4 & [\ 987720,\ 374620,\ 48370,\ 2600,\ 50\ ]\\
6 & 5 & [\ 6271920,\ 2743560,\ 433590,\ 31845,\ 1110,\ 15\ ]\\
6 & 6 & [\ 14566320,\ 7068684,\ 1263364,\ 109245,\ 4915,\ 111,\ 1\ ]\\
6 & 7 & [\ 9755280,\ 5136720,\ 986674,\ 91455,\ 4405,\ 105,\ 1\ ]\\
6 & 8 & [\ 1946520,\ 1108620,\ 221760,\ 20505,\ 900,\ 15\ ]\\
6 & 9 & [\ 156360,\ 98860,\ 20230,\ 1700,\ 50\ ]\\
6 & 10 & [\ 5490,\ 4225,\ 855,\ 50\ ]\\
6 & 11 & [\ 72,\ 93,\ 15\ ]\\
6 & 12 & [\ 0,\ 1\ ]\\
\hline 
7 & 1 & [\ 7,\ 1\ ]\\
7 & 2 & [\ 2296,\ 469,\ 21\ ]\\
7 & 3 & [\ 189714,\ 50589,\ 4116,\ 105\ ]\\
7 & 4 & [\ 5798310,\ 1881530,\ 212765,\ 10150,\ 175\ ]\\
7 & 5 & [\ 74094720,\ 27846910,\ 3904355,\ 260225,\ 8365,\ 105\ ]\\
7 & 6 & [\ 405805680,\ 170232678,\ 27602169,\ 2252775,\ 98910,\ 2247,\ 21\ ]\\
7 & 7 & [\ 892563840,\ 406170120,\ 72388246,\ 6638779,\ 341320,\ 9940,\ 154,\ 1\ ]\\
7 & 8 & [\ 612768240,\ 296719296,\ 56041608,\ 5434744,\ 295365,\ 9079,\ 147,\ 1\ ]\\
7 & 9 & [\ 134802360,\ 68845350,\ 13437879,\ 1314600,\ 69300,\ 1890,\ 21\ ]\\
7 & 10 & [\ 12867120,\ 6962060,\ 1382990,\ 130375,\ 5950,\ 105\ ]\\
7 & 11 & [\ 594090,\ 350280,\ 70175,\ 5880,\ 175\ ]\\
7 & 12 & [\ 13146,\ 9282,\ 1827,\ 105\ ]\\
7 & 13 & [\ 112,\ 133,\ 21\ ]\\
7 & 14 & [\ 0,\ 1\ ]\\
\hline 
8 & 1 & [\ 8,\ 1\ ]\\
8 & 2 & [\ 3984,\ 708,\ 28\ ]\\
8 & 3 & [\ 517720,\ 120092,\ 8624,\ 196\ ]\\
8 & 4 & [\ 25914336,\ 7352044,\ 738038,\ 31556,\ 490\ ]\\
8 & 5 & [\ 579902400,\ 191868740,\ 24093090,\ 1454250,\ 42630,\ 490\ ]\\
8 & 6 & [\ 6105536640,\ 2271563952,\ 333946144,\ 25135180,\ 1032220,\ 22148,\ 196\ ]\\
8 & 7 & [\ 29965844160,\ 12181569792,\ 2001174392,\ 174107332,\ 8743280,\ 255808,\ 4088,\ 28\ ]\\
8 & 8 & [\ 63223332480,\ 27451001136,\ 4862287996,\ 462635796,\ 25989929,\ 886704,\ 18074,\ 204,\ 1\ ]\\
8 & 9 & [\ 44191728000,\ 20147946672,\ 3743518540,\ 373361884,\ 21969689,\ 785008,\ 16730,\ 196,\ 1\ ]\\
8 & 10 & [\ 10460701440,\ 4951790256,\ 945440832,\ 95549692,\ 5573400,\ 189784,\ 3528,\ 28\ ]\\
8 & 11 & [\ 1136459520,\ 555549120,\ 107378264,\ 10671500,\ 581700,\ 16660,\ 196\ ]\\
8 & 12 & [\ 64033200,\ 32382980,\ 6268570,\ 590730,\ 27230,\ 490\ ]\\
8 & 13 & [\ 1952832,\ 1039500,\ 199766,\ 16548,\ 490\ ]\\
8 & 14 & [\ 31080,\ 18676,\ 3472,\ 196\ ]\\
8 & 15 & [\ 208,\ 188,\ 28\ ]\\
8 & 16 & [\ 0,\ 1\ ]\\
\hline
\earr\]
\normalsize
\caption{Polynomials $f_{e,i}(x)$ for $6\leq e \leq 8$ and $1\leq i \leq 2e$}
\label{tbl5}
\end{table}

\end{document}